\documentclass{article}

\usepackage[a4paper, margin=25mm]{geometry}
\usepackage[T1]{fontenc} 
\usepackage{lmodern}

\usepackage{amsmath, amssymb, amsthm, mathtools, bm, tikz}
\usepackage{xcolor}
\usepackage{xpatch}
\usepackage[shortlabels,inline]{enumitem}

\usepackage[
	sorting=nyt,
	maxbibnames=99,
	backref,
	backrefstyle=none, 
	style=alphabetic,
    maxalphanames=5, 
    minalphanames=1   
]{biblatex}

\usepackage[
	colorlinks,
	linkcolor={red!80!black},
	citecolor={blue!80!black},
	urlcolor={green!80!black}
]{hyperref}
\usepackage[capitalise]{cleveref}

\newcommand{\footremember}[2]{%
    \footnote{#2}
    \newcounter{#1}
    \setcounter{#1}{\value{footnote}}%
}
\newcommand{\footrecall}[1]{%
    \footnotemark[\value{#1}]%
}

\setlist[enumerate,1]{label={(\roman*)}}

\numberwithin{equation}{section}
\newcommand\numberthis{\addtocounter{equation}{1}\tag{\theequation}}

\makeatletter
\NewCommandCopy\latexparagraph\paragraph
\RenewDocumentCommand{\paragraph}{sO{#3}m}{%
  \IfBooleanTF{#1}
    {\latexparagraph*{\maybe@addperiod{#3}}}
    {\latexparagraph[#2]{\maybe@addperiod{#3}}}%
}
\newcommand{\maybe@addperiod}[1]{%
  #1\@addpunct{.}%
}
\makeatother

\tikzset{every loop/.style={}}
\tikzset{inline vertex/.style={draw, circle, fill=black, minimum size=2pt, inner sep=0pt}}

\newcommand{\indepg}{\!%
    \begin{tikzpicture}[nodes=inline vertex]
        \node (A) at (0,0) {};
        \node (B) at (0.42,0) {};
        \draw[looseness=15] (A) to[out=125, in=55] (A);
        \draw (A) -- (B);
    \end{tikzpicture}%
    \:\!%
}

\newcommand{\triloop}{\!%
    \begin{tikzpicture}[nodes=inline vertex]
        \node (A) at (0,0) {};
        \node (B) at (0.42,0) {};
        \node (C) at (0.21,0.15) {};
        \draw[looseness=15] (A) to[out=125, in=55] (A);
        \draw (A) -- (B) -- (C) -- (A);
    \end{tikzpicture}%
    \:\!%
}

\DeclareFieldFormat*{title}{\mkbibemph{#1\isdot}}

\DeclareFieldFormat{url}{\href{#1}{\textsc{link}}}
\DeclareFieldFormat{doi}{\href{https://dx.doi.org/#1}{\textsc{doi}}}
\DeclareFieldFormat{eprint:arxiv}{\href{https://arxiv.org/abs/#1}{arXiv:\allowbreak{#1}}}
\DeclareFieldFormat{eprint}{\href{https://arxiv.org/abs/#1}{arXiv:\allowbreak{#1}}}
\AtEveryBibitem{\clearfield{issn}\clearfield{isbn}\clearlist{language}}

\renewbibmacro*{doi+eprint+url}{%
	\printfield{doi}%
	\newunit\newblock%
	\iftoggle{bbx:eprint}{%
		\usebibmacro{eprint}%
	}{}%
	\newunit\newblock%
	\iffieldundef{doi}{%
    \iffieldundef{eprint}{%
		\usebibmacro{url+urldate}%
	}{}}{}%
}

\NewBibliographyString{toappear}
\DefineBibliographyStrings{english}{%
	toappear = {to appear},
}

\renewbibmacro*{in:}{%
	\iffieldundef{pubstate}
	{}
	{\printfield{pubstate}%
		\setunit{\addspace}%
		\clearfield{pubstate}}%
}

\DefineBibliographyStrings{english}{
	backrefpage={},
	backrefpages={}
}

\DeclareFieldFormat{backrefparens}{\addperiod{\scriptsize{#1}}}
\xpatchbibmacro{pageref}{parens}{backrefparens}{}{}

\theoremstyle{plain}
\newtheorem{theorem}{Theorem}[section]
\newtheorem{proposition}[theorem]{Proposition}
\newtheorem{lemma}[theorem]{Lemma}
\newtheorem{conjecture}[theorem]{Conjecture}
\newtheorem{claim}[theorem]{Claim}

\theoremstyle{definition}
\theoremstyle{remark} 
\newtheorem*{remark}{Remark}

\AddToHook{env/proposition/begin}{\crefalias{theorem}{proposition}}
\AddToHook{env/lemma/begin}{\crefalias{theorem}{lemma}}
\AddToHook{env/conjecture/begin}{\crefalias{theorem}{conjecture}}
\AddToHook{env/claim/begin}{\crefalias{theorem}{claim}}

\makeatletter
\newenvironment{claimproof}[1][Proof]{\par
    \pushQED{\qed}%
    
    \normalfont \topsep6\p@\@plus6\p@\relax
    \trivlist
    \item[\hskip\labelsep
    \textit{#1\@addpunct{.}}~]\ignorespaces
}{%
    \popQED\endtrivlist\@endpefalse
}
\makeatother

\makeatletter
\newcommand{\step}[1]{%
  \par
  \addvspace{\medskipamount}
  \noindent%
  \textbf{#1\@addpunct{.}}\enspace\ignorespaces
}
\makeatother

\DeclarePairedDelimiter{\abs}{\lvert}{\rvert}

\newcommand{\defeq}{\coloneqq}
\newcommand{\eqdef}{\eqqcolon}

\newcommand{\RR}{\mathbb{R}}
\newcommand{\calI}{\mathcal{I}}
\newcommand{\calS}{\mathcal{S}}

\newcommand{\blambda}{\bm{\lambda}}
\newcommand{\bmu}{\bm{\mu}}
\newcommand{\one}{\mathbf{1}}

\newcommand{\cart}{\mathbin{\square}} 
\newcommand{\diff}{\mathop{}\!d} 

\newcommand{\Aletheia}{\emph{Aletheia}}

\title{Lower bounds for multivariate independence polynomials and their generalisations}
\author{Joonkyung Lee%
    \footremember{Yonsei}{
        Department of Mathematics, Yonsei University, Seoul, South Korea. Research supported by Samsung STF Grant SSTF-BA2201-02 and the National Research Foundation of Korea (NRF) grant MSIT NRF-2022R1C1C1010300. Email: \texttt{\{joonkyunglee, jaehyeonseo\}@yonsei.ac.kr}. Some mathematical content of this paper was  generated by a custom mathematics research agent powered by Gemini Deep Think. Further documentation, including raw prompts and outputs, can be found in a \href{https://github.com/google-deepmind/superhuman/tree/main/aletheia}{GitHub repository}.
        A Lean formalisation of~\Cref{thm:semiprop-multiaff-lower-bd} is available in the second author's \href{https://github.com/thegreatseo/multivar-indep-formalize}{GitHub repository}.
    }
    \and 
    Jaehyeon Seo\footrecall{Yonsei}}
\date{}

\begin{document}

\maketitle

\begin{abstract}
    In statistical physics, the multivariate hard-core model describes a system of particles, each of which receives its own fugacity.
    In graph-theoretic language, the partition function of the model translates to the \emph{multivariate independence polynomial}, i.e., the multiaffine generalisation of the independence polynomial, defined by
    $Z_G(\lambda_1,\dots,\lambda_n) \defeq \sum_{I\in\mathcal{I}(G)} \prod_{v\in I}\lambda_v$,
    where $\mathcal{I}(G)$ denotes the set of all independent sets in a graph $G$ on $[n]\defeq\{1,2,\dots,n\}$.
    We prove that for every simple graph $G$ on $[n]$ and $\lambda_1,\dots,\lambda_n\geq 0$,
    \[
        Z_G(\lambda_1,\dots,\lambda_n) \geq \prod_{i=1}^n (1+(d_i+1)\lambda_i)^{1/(d_i+1)},
    \]
    where $d_1,\dots,d_n$ is the degree sequence of $G$. This generalises a result of Sah, Sawhney, Stoner, and Zhao, who proved the univariate case $\lambda_1=\dots=\lambda_n=\lambda$. 
    
    We further conjecture that our inequality should generalise to other antiferromagnetic models and give some evidence in support of it. In particular, for $\lambda_i,\mu_i\geq 0$, $1\leq i\leq n$, we obtain a stronger inequality
    \[
        \sum_{\substack{I,J\in \mathcal{I}(G) \\ I\cap J=\emptyset}} \prod_{v\in I}\lambda_v\prod_{u\in J}\mu_u
        \geq \prod_{i=1}^n \left(1+(d_i+1)(\lambda_i+\mu_i)+d_i(d_i+1)\lambda_i\mu_i\right)^{1/(d_i+1)},
    \]
    which proves our conjecture for a multiaffine generalisation of the \emph{semiproper colouring} partition function with two proper colours. The proof is formalised in the Lean~4 proof assistant using Harmonic Aristotle.

    Our key technical steps for both theorems were obtained by using a custom mathematical research agent built on top of Gemini Deep Think, which can be seen as a benchmark demonstrating that the current state-of-the-art language models can, in part, assist with mathematical research.
\end{abstract}

\section{Introduction}
Independence polynomials are a central object in both combinatorics and statistical physics. 
In particular, its multiaffine generalisation, the \emph{multivariate independence polynomial}, defined by
\[
    Z_G(\blambda) = \sum_{I\in \calI(G)}\prod_{v\in I}\lambda_v,
\]
where $\blambda=(\lambda_v)_{v\in V(G)}$ and $\calI(G)$ denotes the set of independent sets in $G$, has played an important role in bridging combinatorics and statistical physics. 
Indeed, it is the \emph{partition function} of the multivariate hard-core model and furthermore, the \emph{cluster expansion}, which relies on the multivariate Taylor expansion of $\log Z_G(\blambda)$, has been an important tool in modern combinatorics; see~\cite{davies2025hardcore,Jenssen_2024} for an overview.

With slight abuse of notation, write $Z_G(\lambda)\defeq Z_G(\lambda,\dots,\lambda)$, the usual single-variate independence polynomial or, equivalently, the partition function of the hard-core model with fugacity $\lambda$.
Our first main result is to give a general lower bound for $Z_G(\blambda)$ that works for any finite simple graph $G$ and any nonnegative reals $\lambda_v$, \(v\in V(G)\).

\begin{theorem}\label{thm:indep-poly-multiaff-lower-bd}
	Let \(G\) be a graph and let \(\blambda=(\lambda_v)\in (\RR_{\geq 0})^{V(G)}\). Write $d_v\defeq \deg_G(v)$. Then
	\begin{equation}\label{eq:main}
	    Z_G(\blambda) \ge \prod_{v \in V(G)} Z_{K_{d_v+1}}(\lambda_v)^{\frac{1}{d_v+1}}
		= \prod_{v \in V(G)} ( (d_v+1)\lambda_v + 1 )^{\frac{1}{d_v+1}}.
	\end{equation}
    Equality holds if and only if on each connected component of \(G\), either
    \begin{enumerate*}
        \item all \(\lambda_v=0\) or 
        \item \(\lambda_v\) is constant and each component of $G$ is a clique.
    \end{enumerate*}
\end{theorem}

This generalises a theorem of Sah, Sawhney, Stoner, and Zhao \cite[Theorem~1.7]{sah2019number}, who proved~\eqref{eq:main} for single-variate independence polynomials, i.e., the case $\lambda_v=\lambda$ for all $v\in V(G)$. We also remark that the simplest case, when $\lambda_v=1$ and $d_v=d$ for all $v\in V(G)$, was first proved by Cutler and Radcliffe~\cite[Corollary~2.3]{cutler2017maximum}.
It is worth noting that, unlike previous results, the inequality~\eqref{eq:main} for $G=K_{d+1}$ is not an equality but an AM--GM-type inequality.

A quick corollary of the inequality~\eqref{eq:main} is a lower bound for the generalised independence polynomial $Z_{G}(\lambda;\alpha)=\sum_{I\in\calI(G)}\lambda^{\abs{I}}\alpha^{e(I,\overline{I})}$,
where $e(I,\overline{I})$ is the number of crossing edges between $I$ and $V(G)\setminus I$, a partition function of a \(2\)-spin \emph{antiferromagnetic} model, which we shall define later.
Then setting $\lambda_v=\lambda \alpha^{d_v}$ in~\eqref{eq:main} gives the following.

\begin{theorem}\label{thm:2spin}
    Let \(G\) be a graph and let \(\lambda\) and $\alpha$ be nonnegative reals. Write $d_v\defeq \deg_G(v)$. Then
	\[
	    Z_G(\lambda;\alpha) \ge \prod_{v \in V(G)} Z_{K_{d_v+1}}(\lambda;\alpha)^{\frac{1}{d_v+1}}
		= \prod_{v \in V(G)} \bigl( (d_v+1)\lambda\alpha^{d_v} + 1 \bigr)^{\frac{1}{d_v+1}}.
	\]
\end{theorem}

We suspect that this result generalises even further. To elaborate on what this would mean, let $H$ be an edge-weighted graph possibly with loops. Let
\[
    \hom (G,H) \defeq \sum_{\phi:V(G)\rightarrow V(H)}\prod_{uv\in E(G)}H(\phi(u),\phi(v))
\]
be the (weighted) homomorphism count for $G$ in $H$, where $H(\cdot,\cdot)$ denotes the corresponding edge weight of $H$.
Then the independence polynomial $Z_G(\lambda)$ for positive rational values $\lambda=q/p$, $p,q\in\mathbb{N}$, corresponds to $\hom(G,H^{\lambda})$, where $H^{\lambda}$ denotes a suitably `blown-up' copy of the graph $H=\indepg$, i.e., replace the two vertices by $p$ self-looped ones that connect to one another and $q$ non-looped ones, respectively, and put the complete bipartite graph $K_{p,q}$ in between. A graph $H$ with nonnegative edge weights is said to be \emph{antiferromagnetic} if $H$, as a symmetric matrix with entries corresponding to the edge weights, has exactly one positive eigenvalue, counting multiplicity. For instance, $H=\indepg$ is antiferromagnetic, and so are its blow-ups $H^\lambda$.
We conjecture that any antiferromagnetic $H$ is a `clique-minimiser' in the following sense.

\begin{conjecture}\label{conj:main}
    Let \(G\) be a graph and write $d_v\defeq \deg_G(v)$. Then for any antiferromagnetic edge-weighted graph $H$,
    \begin{equation}\label{eq:afm}
        \hom(G,H)\geq \prod_{v\in V(G)}\hom(K_{d_v+1},H)^{\frac{1}{d_v+1}}.
    \end{equation}
\end{conjecture}

For instance, the Sah--Sawhney--Stoner--Zhao inequality is the aforementioned case $H^{\lambda}$ with $H=\indepg$. It is not hard to settle the case $H=K_q$, by a slight generalisation of Csikv\'{a}ri's inequality, recorded in Zhao's survey~\cite[Theorem~8.3]{zhao2017extremal}.
A particular case of~\Cref{conj:main} when $G$ is $d$-regular $H$ represents an antiferromagnetic \emph{Ising model}, i.e., the $2\times 2$ matrix with $H_{11}=H_{22}=B\in (0,1)$ and $H_{12}=H_{21}=1$, was conjectured by Davies an Leblanc~\cite[Conjecture~5(ii)]{DL24}.
It can also be seen as an antiferromagnetic counterpart to another theorem of Sah, Sawhney, Stoner, and Zhao \cite[Theorem~1.14]{sah2020reverse}, stating that the reverse inequality to~\eqref{eq:afm} holds whenever $H$ is \emph{ferromagnetic}, i.e., all the eigenvalues are nonnegative. 

Recall that \Cref{thm:2spin} makes partial progress towards~\Cref{conj:main}. We make further progress in multiple directions. First of all, we confirm the conjecture for $K_3^{\circ}\defeq\triloop$, one loop added to a triangle, in a stronger form that generalises~\Cref{thm:indep-poly-multiaff-lower-bd}.
To state the result, for $\blambda,\bmu\in\RR^{V(G)}$, let
\[
    Z_G^{(2)}(\blambda,\bmu)\defeq \sum_{\substack{I,J\in \calI(G) \\ I\cap J=\emptyset}} \prod_{v\in I}\lambda_v\prod_{u\in J}\mu_u
\]
and write $Z_G^{(2)}(\lambda,\mu)=Z_G^{(2)}(\lambda\one,\mu\one)$, where $\one=(1,1,\dots,1)$.
In particular, $Z_G^{(2)}(1,1)=\hom(G,K_3^{\circ})$. This corresponds to the number of \emph{semiproper colourings} with two proper colours, i.e., the vertex colouring using two `proper' colours that cannot make an edge monochromatic and one `free' colour that can be used with no restrictions.
Furthermore, $Z_G^{(2)}(\lambda,\mu)$, $\lambda,\mu>0$, counts $\hom(G,H)$, where $H$ is obtained by blowing up vertices of $K_3^{\circ}$ in an appropriate way. For instance, adjusting $\lambda$ and $\mu$ gives a semiproper colouring with two proper colours and multiple free colours as well.
Our second main result is the following.
\begin{theorem}\label{thm:semiprop-multiaff-lower-bd}
    Let \(G\) be a graph and let \(\lambda_v\) and \(\mu_v\), \(v\in V(G)\), be nonnegative reals. Then
	\begin{equation}\label{eq:main2}
	    Z_G^{(2)}(\blambda,\bmu) \ge \prod_{v \in V(G)} Z_{K_{d_v+1}}^{(2)}(\lambda_v,\mu_v)^{\frac{1}{d_v+1}}
		= \prod_{v \in V(G)} \big( (d_v+1)d_v\lambda_v\mu_v + (d_v+1)(\lambda_v +\mu_v) + 1 \big)^{\frac{1}{d_v+1}}.
    \end{equation}
\end{theorem}
Although letting $\mu_v=0$ for all $v\in V(G)$ reproves~\Cref{thm:indep-poly-multiaff-lower-bd}, we shall give independent proofs for each of the results. 

As an application of our main results, we give lower bounds slightly weaker than~\eqref{eq:afm} for semiproper colouring partition functions with $q$ proper colours, $q>2$, which can be seen as an `approximate version' of~\Cref{conj:main} for semiproper colourings; see \Cref{thm:semiproper-colouring-poly-multiaff-lower-bd-weak}. Moreover, we also settle the conjecture for the case $\Delta(G)\leq 2$, where \(\Delta(G)\) is the maximum degree of \(G\); see~\Cref{thm:conj-main-deg<=2}.

An interesting feature of~\Cref{thm:indep-poly-multiaff-lower-bd,thm:semiprop-multiaff-lower-bd} is that the key steps in both proofs were obtained, at least in part, by using \Aletheia, a mathematics research agent built upon Gemini Deep Think at Google DeepMind, under the lead of Tony Feng.
We will point out these points of contribution where they arise. 
We also refer to~\cite{aletheia} for more details of the development of {\Aletheia} and its applications to various problems.

We also formalised the proof of~\Cref{thm:semiprop-multiaff-lower-bd} in the Lean~4 proof assistant using Harmonic Aristotle~\cite{achim2025aristotleimolevelautomatedtheorem}. The resulting Lean files are available in the second author's GitHub repository~\cite{seo2026multivariate}.

\section{Multivariate independence polynomials}\label{sec:indep-poly}

The starting point of our proof of~\Cref{thm:indep-poly-multiaff-lower-bd} is the standard vertex-deletion recurrence for the independence polynomial. Namely, for a graph \(G\) and its vertex \(w\), conditioning on whether \(w\) is occupied,
\[
    Z_G(\lambda) = Z_{G\setminus w}(\lambda) + \lambda Z_{G\setminus (\{w\}\cup N(w))}(\lambda),
\]
where \(N(w)\) is the set of neighbours of \(w\) in \(G\). Equivalently, in the multivariate setting, deleting \(N(w)\) is the same as setting \(\lambda_v=0\) for each \(v\in N(w)\).
This recurrence and its variants have been used in several places to bound the number of independent sets~\cite{galvin2011number,sah2019number,sah2020reverse,buys2025triangle}, and it also underlies inductive lower bounds on the independence number \(\alpha(G)\) (see, e.g., \cite{shearer1983note}).

\medskip

Let us formalise this inductive idea. First, both sides of~\eqref{eq:main} factor over connected components of~\(G\), so we may assume that \(G\) is connected.
For a graph \(G\) and \(\blambda=(\lambda_v)_{v\in V(G)}\), let
\[
	Z^-_G(\blambda) \defeq \prod_{v \in V(G)} Z_{K_{d_v+1}}(\lambda_v)^{\frac{1}{d_v+1}}.
\]
Then \Cref{thm:indep-poly-multiaff-lower-bd} states that \(Z_G(\blambda) \ge Z^-_G(\blambda)\) for all \(\blambda\in (\RR_{\ge0})^{V(G)}\).
We use induction on \(\abs{V(G)}\) to prove this inequality, where the base case \(\abs{V(G)}=1\) is clear. 

Let \(w\) be a vertex in \(G\) with maximum degree \(\Delta\). Fix \(\blambda=(\lambda_v)_{v\in V(G)}\in (\RR_{\ge0})^{V(G)}\) and let \(\blambda'\defeq (\lambda_v)_{v\in V(G \setminus w)}\). We begin with the multivariate recursive expression
\[
	Z_G(\blambda) = Z_{G \setminus w}(\blambda') + \lambda_w Z_{G \setminus w}(\blambda')|_{\lambda_v = 0,\, v \in N(w)}.
\]
By the induction hypothesis, it suffices to show that
\begin{equation}\label{eq:indep-multiaff-lower-bd-goal}
	Z^-_{G \setminus w}(\blambda') + \lambda_w Z^-_{G\setminus w}(\blambda')|_{\lambda_v = 0, v \in N(w)}
	\ge Z^-_G(\blambda).
\end{equation}
The definition gives
\[
	Z^-_{G \setminus w}(\blambda')
	= \prod_{v\in N(w)} Z_{K_{d_v}}(\lambda_v)^{\frac{1}{d_v}}
		\prod_{v \in V(G)\setminus (\{w\}\cup N(w))} Z_{K_{d_v+1}}(\lambda_v)^{\frac{1}{d_v+1}},
\]
so after removing the common factor \(\prod\limits_{v\in V(G)\setminus (\{w\}\cup N(w))} Z_{K_{d_v+1}}(\lambda_v)^{\frac{1}{d_v+1}}\), \eqref{eq:indep-multiaff-lower-bd-goal} is equivalent to
\begin{equation}\label{eq:indep-multiaff-lower-bd-goal-2}
	\prod_{v\in N(w)} Z_{K_{d_v}}(\lambda_v)^{\frac{1}{d_v}}
		+ \lambda_w
	\ge Z_{K_{\Delta+1}}(\lambda_w)^{\frac{1}{\Delta+1}}
		\prod_{v\in N(w)} Z_{K_{d_v+1}}(\lambda_v)^{\frac{1}{d_v+1}},
\end{equation}
which is shown by the following \lcnamecref{lem:key-tech-ineq}.
We remark that this result is exactly what {\Aletheia} successfully derived. We integrated this with the correct solutions provided by other models, including Gemini Deep Think 3.0 and its IMO-Gold version, to formulate a cleaner argument.

\begin{lemma}\label{lem:key-tech-ineq}
    Let \(\Delta\ge 1\) be an integer and \(1\le d_1,\dots,d_\Delta\le \Delta\). Let \(\lambda_0,\lambda_1,\dots,\lambda_\Delta\) be nonnegative reals. Then
    \[
        \prod_{i=1}^\Delta (1+d_i \lambda_i)^{\frac{1}{d_i}} + \lambda_0
        \ge (1+(\Delta+1)\lambda_0)^{\frac{1}{\Delta+1}} \prod_{i=1}^\Delta (1+(d_i+1)\lambda_i)^{\frac{1}{d_i+1}}.
    \]
\end{lemma}

\begin{proof}
    Let \(A\defeq \prod_{i=1}^\Delta A_i\) and \(B\defeq \prod_{i=1}^\Delta B_i\) where
    \[
        A_i \defeq (1+d_i\lambda_i)^{\frac{1}{d_i}}
        \quad\text{and}\quad
        B_i \defeq (1+(d_i+1)\lambda_i)^{\frac{1}{d_i+1}}.
    \]
    The desired inequality then rephrases as
    \begin{equation}\label{eq:goal}
        A + \lambda_0 \ge (1 + (\Delta+1)\lambda_0)^{\frac{1}{\Delta+1}} B.
    \end{equation}
    By the AM--GM inequality,
    \begin{align*}
        (1 + (\Delta+1)\lambda_0)^{\frac{1}{\Delta+1}} B
        &= \bigl( (1 + (\Delta+1)\lambda_0) \cdot (B^{\frac{\Delta+1}{\Delta}})^\Delta \bigr)^{\frac{1}{\Delta+1}}\\
        &\le \frac{(1 + (\Delta+1)\lambda_0) + \Delta B^{\frac{\Delta+1}{\Delta}}}{\Delta+1}
        = \frac{1 + \Delta B^{\frac{\Delta+1}{\Delta}}}{\Delta+1} + \lambda_0,
    \end{align*}
    so it suffices to show that
    \begin{equation}\label{eq:key-tech-ineq-reduced}
        (\Delta+1)A \ge \Delta B^{\frac{\Delta+1}{\Delta}} + 1.
    \end{equation}
    The proof of this inequality consists of two technical steps. Firstly, we prove that, for each \(1\le i\le \Delta\),
    \begin{equation}\label{eq:indep-multiaff-lower-bd-component}
         (\Delta+1)A_i^{\Delta}=(\Delta+1)(1+d_i\lambda_i)^{\frac{\Delta}{d_i}} \ge 1+ \Delta(1+(d_i+1)\lambda_i)^{\frac{\Delta+1}{d_i+1}}=1+\Delta B_i^{\Delta+1},
    \end{equation}
    multiplying which for all \(1\le i\le \Delta\) gives
    \begin{equation}\label{eq:key-tech-intermediate}
        (\Delta+1)^\Delta A^\Delta \ge \prod_{i=1}^\Delta (1 + \Delta B_i^{\Delta+1}).
    \end{equation}
    The second step is to show that 
    \begin{equation}\label{eq:key-tech-lem-goal-2}
        \prod_{i=1}^\Delta (1 + \Delta B_i^{\Delta+1})^{\frac{1}{\Delta}}
        \ge
        1+\Delta B^{\frac{\Delta+1}{\Delta}}.
    \end{equation}
    Combining this with~\eqref{eq:key-tech-intermediate} concludes the proof of~\eqref{eq:key-tech-ineq-reduced}.

    To prove the first inequality~\eqref{eq:indep-multiaff-lower-bd-component}, the following claim with \(d=d_i\) and \(\lambda=\lambda_i\) suffices.
    \begin{claim}\label{clm:key-tech-lem-sub-1}
        For \(1\le d\le \Delta\) and \(\lambda\ge0\),
        \[
            (\Delta+1) (1+d\lambda)^{\Delta/d} \ge 1 + \Delta (1+(d+1)\lambda)^{(\Delta+1)/(d+1)}.
        \]
    \end{claim}
    \begin{claimproof}
        Let
        \[
            f(\lambda) \defeq (\Delta+1) (1+d\lambda)^{\Delta/d} - \Delta (1+(d+1)\lambda)^{(\Delta+1)/(d+1)} - 1,
        \]
        then we want to show \(f(\lambda)\ge0\).
        Since \(f(0)=0\), it suffices to show \(f'(\lambda)\ge0\) for all \(\lambda\ge0\). We have
        \begin{align*}
            f'(\lambda)
            &= \Delta(\Delta+1) (1+d\lambda)^{\Delta/d-1} - \Delta(\Delta+1)(1+(d+1)\lambda)^{(\Delta+1)/(d+1)-1}
            \\&= \Delta(\Delta+1) \bigl( (1+d\lambda)^{(\Delta-d)/d} - (1+(d+1)\lambda)^{(\Delta-d)/(d+1)} \bigr),
        \end{align*}
        which is nonnegative if
        \[
            (1+d\lambda)^{(d+1)/d} \ge 1+(d+1)\lambda.
        \]
        This is a consequence of Bernoulli's inequality \((1+x)^\alpha\ge 1+\alpha x\), for \(x\ge0\) and \(\alpha\ge1\).
    \end{claimproof}
    It remains to show~\eqref{eq:key-tech-lem-goal-2}. Let $y_i\defeq\Delta B_i^{\Delta+1}$. Then~\eqref{eq:key-tech-lem-goal-2} is equivalent to 
    \[
        \prod_{i=1}^\Delta (1+y_i)^{1/\Delta} \ge 1+\prod_{i=1}^\Delta y_i^{1/\Delta}.
    \]
    This follows from the AM--GM inequality of the form
    \[
        1
        = \frac{1}{\Delta} \biggl( \sum_{i=1}^\Delta \frac{1}{1+y_i} + \sum_{i=1}^\Delta \frac{y_i}{1+y_i} \biggr)
        \ge \prod_{i=1}^\Delta \frac{1}{(1+y_i)^{1/\Delta}} + \prod_{i=1}^\Delta \frac{y_i^{1/\Delta}}{(1+y_i)^{1/\Delta}},
    \]
     Indeed, multiplying $\prod_{i=1}^\Delta (1+y_i)^{1/\Delta}$ on both sides concludes the proof.
\end{proof}

\begin{remark}
    It is not hard to see that equality in~\Cref{lem:key-tech-ineq} holds if and only if either
    \begin{enumerate}[(i)]
        \item $\lambda_0=\lambda_1=\dots=\lambda_\Delta=0$ or
        \item $\lambda_0=\lambda_1=\dots=\lambda_\Delta\neq 0$ and $d_1=\dots=d_\Delta =\Delta$.
    \end{enumerate}
    Using this, it follows that equality in~\eqref{eq:main} holds if and only if for each connected component \(C\) of \(G\),
    \begin{enumerate}[(a)]
        \item $\lambda_v=0$ for all $v\in C$ or
        \item $\lambda_v=\lambda\neq 0$ for all $v\in C$ and $G[C]$ is a clique.
    \end{enumerate}
    To see this, we use induction on \(\abs{V(G)}\), where the base case is clear. We may assume that \(G\) is connected. Suppose equality in~\eqref{eq:main} holds so that \(Z_G(\blambda) = Z^-_G(\blambda)\). Recall that \(Z_G(\blambda) \ge Z^-_G(\blambda)\) follows from
    \begin{align}
        Z_G(\blambda)
        &= Z_{G \setminus w}(\blambda') + \lambda_w Z_{G \setminus w}(\blambda')|_{\lambda_v = 0,\, v \in N(w)}
            \nonumber
        \\&\ge Z^-_{G \setminus w}(\blambda') + \lambda_w Z^-_{G \setminus w}(\blambda')|_{\lambda_v = 0,\, v \in N(w)}
            \label{eq:equality-cond-ineq-1}
        \\&\ge Z^-_G(\blambda),
            \label{eq:equality-cond-ineq-2}
    \end{align}
    and that \eqref{eq:equality-cond-ineq-2} is equivalent to \eqref{eq:indep-multiaff-lower-bd-goal-2}, where the latter holds by~\Cref{lem:key-tech-ineq}.
    Equality occurs in~\eqref{eq:equality-cond-ineq-1} if and only if both of the following hold:
    \begin{enumerate}[label=(A\arabic*)]
        \item \(Z_{G \setminus w}(\blambda') = Z^-_{G \setminus w}(\blambda')\).
        \label{item:eq-cond-A1}
        \item \(\lambda_w=0\) or \(Z^-_{G \setminus w}(\blambda')|_{\lambda_v = 0,\, v \in N(w)} = Z_{G \setminus w}(\blambda')|_{\lambda_v = 0,\, v \in N(w)}\).
        \label{item:eq-cond-A2}
    \end{enumerate}
    \ref{item:eq-cond-A1} occurs, by the induction hypothesis, only if \(\lambda_v\) is constant on each connected component of \(G \setminus w\).
    In addition, \ref{item:eq-cond-A2} occurs, again by the induction hypothesis, if and only if either
    \begin{enumerate*}[label=(B\arabic*)]
        \item \(\lambda_w=0\) or
        \label{item:eq-cond-B1}
        \item \(\lambda_v=0\) for all \(v\in V(G)\setminus  (\{w\}\cup N(w))\).
        \label{item:eq-cond-B2}
    \end{enumerate*}
    
    Equality occurs in~\eqref{eq:equality-cond-ineq-2}, by the equality case of~\eqref{eq:indep-multiaff-lower-bd-goal-2}, if and only if either
    \begin{enumerate}[label=(C\arabic*)]
        \item \(\lambda_w=\lambda_v=0\) for all \(v\in N(w)\) or
        \label{item:eq-cond-C1}
        \item \(\lambda_w=\lambda_v\) and \(d_v=\Delta\) for all \(v\in N(w)\).
        \label{item:eq-cond-C2}
    \end{enumerate}
    In either of the cases, since \(G\) is connected and \(\lambda_v\) is constant on each connected component of \(G \setminus w\), \(\lambda_v\) is constant on \(V(G)\).
    If~\ref{item:eq-cond-C1} occurs, \(\lambda_v=0\) for all \(v\in V(G)\).
    Otherwise, suppose \ref{item:eq-cond-C2} occurs but \(\lambda_w=\lambda_v\neq 0\) for all \(v\in N(w)\). Then \ref{item:eq-cond-B1} is not the case, so \ref{item:eq-cond-B2} must occur. This implies that \(V(G)\setminus (\{w\}\cup N(w))\) is empty, since otherwise \(\lambda_v=0\) for all \(v\in V(G)\). Thus, \(V(G) = \{w\}\cup N(w)\), and since \(d_v=\Delta\) for all \(v\in N(w)\), \(G\) must be a clique.
\end{remark}

\section{Semiproper colourings}\label{sec:semiproper}
In our proof of \Cref{thm:semiprop-multiaff-lower-bd}, we used {\Aletheia} in a more complex and rather interactive way than in the proof of~\Cref{thm:indep-poly-multiaff-lower-bd}, where we essentially outsourced one key lemma. Because of this, we shall first give a full proof and then discuss how the AI agent helped along the way in the remark at the very end of this section.

To prove \Cref{thm:semiprop-multiaff-lower-bd}, we may assume that \(G\) is connected, analogously to the proof of~\Cref{thm:indep-poly-multiaff-lower-bd}. For a graph \(G\) and \(\blambda=(\lambda_v)_{v\in V(G)}\), \(\bmu=(\mu_v)_{v\in V(G)}\), let
\[
    Z_G^{(2)-}(\blambda,\bmu) \defeq \prod_{v\in V(G)} Z^{(2)}_{K_{d_v+1}}(\lambda_v,\mu_v)^{\frac{1}{d_v+1}}.
\]
Then \Cref{thm:semiprop-multiaff-lower-bd} states that \(Z^{(2)}_G(\blambda,\bmu)\ge Z_G^{(2)-}(\blambda,\bmu)\) for all \(\blambda,\bmu\in (\RR_{\ge0})^{V(G)}\). To prove this inequality, we again use the idea of `conditioning on the state assigned to \(w\)' to obtain a recurrence. 

To elaborate, we proceed by induction on \(\abs{V(G)}\), where the base case \(\abs{V(G)}=1\) holds clearly.
Let~\(w\) be a vertex in \(G\) with maximum degree \(\Delta\). Fix \(\blambda,\bmu\in (\RR_{\ge0})^{V(G)}\) and let \(\blambda'\defeq (\lambda_v)_{v\in V(G\setminus w)}\) and \(\bmu'\defeq (\mu_v)_{v\in V(G\setminus w)}\). The recurrence relation obtained by fixing the image of $w$ in $K_3^\circ$ is
\[
    Z^{(2)}_G(\blambda,\bmu)
    = Z^{(2)}_{G\setminus w}(\blambda',\bmu')
        + \lambda_w Z^{(2)}_{G\setminus w}(\blambda',\bmu')|_{\lambda_v=0,\, v\in N(w)}
        + \mu_w Z^{(2)}_{G\setminus w}(\blambda',\bmu')|_{\mu_v=0,\, v\in N(w)}.
\]
Using the induction hypothesis, it suffices to show that
\begin{equation}\label{eq:semiprop-multiaff-lower-bd-goal}
    Z^{(2)-}_{G\setminus w}(\blambda',\bmu')
        + \lambda_w Z^{(2)-}_{G\setminus w}(\blambda',\bmu')|_{\lambda_v=0,\, v\in N(w)}
        + \mu_w Z^{(2)-}_{G\setminus w}(\blambda',\bmu')|_{\mu_v=0,\, v\in N(w)}
    \ge Z^{(2)-}_G(\blambda,\bmu).
\end{equation}
By definition,
\[
	Z^{(2)-}_{G \setminus w}(\blambda',\bmu')
	= \prod_{v\in N(w)} Z^{(2)}_{K_{d_v}}(\lambda_v,\mu_v)^{\frac{1}{d_v}}
		\prod_{v \in V(G)\setminus (\{w\}\cup N(w))} Z^{(2)}_{K_{d_v+1}}(\lambda_v,\mu_v)^{\frac{1}{d_v+1}},
\]
so after removing the common factors \(Z^{(2)}_{K_{d_v+1}}(\lambda_v,\mu_v)^{\frac{1}{d_v+1}}\) for \(v\in V(G)\setminus (\{w\}\cup N(w))\), \eqref{eq:semiprop-multiaff-lower-bd-goal} becomes
\begin{multline}\label{eq:semiprop-multiaff-lower-bd-goal-2}
    \prod_{v\in N(w)} A_{d_v}(\lambda_v,\mu_v)^{\frac{1}{d_v}}
        + \lambda_w \prod_{v\in N(w)} B_{d_v}(\mu_v)^{\frac{1}{d_v}}
        + \mu_w \prod_{v\in N(w)} B_{d_v}(\lambda_v)^{\frac{1}{d_v}} \\
    \ge A_{\Delta+1}(\lambda_w,\mu_w)^{\frac{1}{\Delta+1}} \prod_{v\in N(w)} A_{d_v+1}(\lambda_v,\mu_v)^{\frac{1}{d_v+1}},
\end{multline}
where for \(d\ge1\) and \(\lambda,\mu\ge0\),
\begin{align*}
    A_d(\lambda,\mu) &\defeq d(d-1)\lambda\mu + d(\lambda+\mu) + 1
        = Z^{(2)}_{K_d}(\lambda,\mu)
    \quad\text{and} \\
    B_d(\lambda) &\defeq d\lambda+1
        = Z^{(2)}_{K_d}(\lambda,0) = Z_{K_d}(\lambda).
\end{align*}

It is convenient to treat the case $\Delta=1$ separately, and then assume $\Delta\geq 2$.
Suppose $\Delta=1$ and let \(v\) be the unique neighbour of \(w\). Since \(d_v=\Delta=1\), \eqref{eq:semiprop-multiaff-lower-bd-goal-2} is written as
\begin{equation}\label{eq:semiprop-multiaff-lower-bd-goal-2-Delta=1}
    A_{d_v}(\lambda_v,\mu_v) + \lambda_w B_{d_v}(\mu_v) + \mu_w B_{d_v}(\lambda_v) \ge A_2(\lambda_w,\mu_w)^{1/2} A_2(\lambda_v,\mu_v)^{1/2}.
\end{equation}
Let \(x_1\defeq \lambda_w+1\), \(y_1\defeq \mu_w+1\), \(x_2\defeq \lambda_v+1\), and \(y_2\defeq \mu_v+1\). Then \eqref{eq:semiprop-multiaff-lower-bd-goal-2-Delta=1} is equivalent to
\[
    x_1y_2 + y_1x_2 - 1 \ge (2x_1y_1-1)^{1/2} (2x_2y_2-1)^{1/2}.
\]
The AM--GM inequality gives
\[
    x_1y_2 + y_1x_2 - 1 \ge 2(x_1y_2y_1x_2)^{1/2} - 1.
\]
Substituting \(x=(x_1y_1)^{1/2}\) and \(y=(x_2y_2)^{1/2}\) into \((2xy-1)^2 - (2x^2-1)(2y^2-1) = 2(x-y)^2 \ge0\) gives
\[
    2(x_1y_1x_2y_2)^{1/2} - 1 \ge (2x_1y_1 - 1)^{1/2} (2x_2y_2 - 1)^{1/2}
\]
and hence \eqref{eq:semiprop-multiaff-lower-bd-goal-2-Delta=1} holds. Therefore, in what follows, we assume $\Delta\geq 2$.

\subsection{The dual set \texorpdfstring{\(\calS_d\)}{S\_d} and log-convexity}
For \(d\ge2\), let
\[
    \calS_d \defeq \{(a_0,a_1,a_2)\in(\RR_{>0})^3 : a_0+a_1x+a_2y\ge A_{d+1}(x,y)^{\frac{1}{d+1}}\ \forall x,y\ge0\}.
\]
One good reason to look into this set is that $A_{d+1}(x,y)^{\frac{1}{d+1}}$ is concave, which means $\calS_d$ may be viewed as a `dual' of it. Concavity of $A_{d+1}^{\frac{1}{d+1}}$ is straightforward to check.

\begin{lemma}\label{lem:Ad+1-concave}
	For \(d\ge2\), \(A_{d+1}(x,y)^{\frac{1}{d+1}}\) is concave on any convex subset of \(\{(x,y)\in\RR^2 : A_{d+1}(x,y)>0\}\).
\end{lemma}
\begin{proof}
    Let \(\Omega\) be a convex subset of \(\{(x,y)\in\RR^2 : A_{d+1}(x,y)>0\}\). A direct calculation gives
    \begin{align}
        \partial_{xx} \bigl( A_{d+1}(x,y)^{\frac{1}{d+1}} \bigr)
        &= -d(dy+1)^2 A_{d+1}(x,y)^{\frac{1}{d+1}-2} \le0,
        \label{eq:Ad+1-concave-pdiff-xx} \\
        \partial_{yy} \bigl( A_{d+1}(x,y)^{\frac{1}{d+1}} \bigr)
        &= -d(dx+1)^2 A_{d+1}(x,y)^{\frac{1}{d+1}-2} \le0,
        \label{eq:Ad+1-concave-pdiff-yy} \\
        \partial_{xy} \bigl( A_{d+1}(x,y)^{\frac{1}{d+1}} \bigr)
        &= d(dxy+x+y) A_{d+1}(x,y)^{\frac{1}{d+1}-2}. \nonumber
    \end{align}
    Then the Hessian determinant of \(A_{d+1}(x,y)^{\frac{1}{d+1}}\) is
    \[
        d^2 A_{d+1}(x,y)^{\frac{2}{d+1}-3} \bigl( (d-1)(dxy+x+y) + 1 \bigr).
    \]
    Since \(A_{d+1}(x,y) = (d+1)(dxy+x+y)+1 > 0\) on \(\Omega\),
    \[
        (d-1)(dxy+x+y) + 1 > -\frac{d-1}{d+1} + 1 > 0,
    \]
    so \(A_{d+1}(x,y)^{\frac{1}{d+1}}\) has a positive Hessian determinant on \(\Omega\). Together with \eqref{eq:Ad+1-concave-pdiff-xx} and \eqref{eq:Ad+1-concave-pdiff-yy}, we conclude that \(A_{d+1}(x,y)^{\frac{1}{d+1}}\) is concave on \(\Omega\).
\end{proof}

Now our goal \eqref{eq:semiprop-multiaff-lower-bd-goal-2} is equivalent to the following \lcnamecref{lem:Sn-membership}.
\begin{lemma}\label{lem:Sn-membership}
    Let \(\Delta\ge2\) and \(1\le d_i\le \Delta\), \(1\le i\le \Delta\) be integers. Let \(\lambda_i,\mu_i\ge0\) for \(1\le i\le \Delta\). Then
    \[
        \biggl(
        \prod_{i=1}^\Delta \frac{A_{d_i}(\lambda_i,\mu_i)^{\frac{1}{d_i}}}{A_{d_i+1}(\lambda_i,\mu_i)^{\frac{1}{d_i+1}}},\,
        \prod_{i=1}^\Delta \frac{B_{d_i}(\mu_i)^{\frac{1}{d_i}}}{A_{d_i+1}(\lambda_i,\mu_i)^{\frac{1}{d_i+1}}},\,
        \prod_{i=1}^\Delta \frac{B_{d_i}(\lambda_i)^{\frac{1}{d_i}}}{A_{d_i+1}(\lambda_i,\mu_i)^{\frac{1}{d_i+1}}}
        \biggr) \in \calS_\Delta.
    \]
\end{lemma}

Our plan is to show that \(\log\calS_\Delta \defeq \{(\log a_0,\log a_1,\log a_2):(a_0,a_1,a_2)\in \calS_\Delta\}\) is convex or, equivalently, for all \((a_0,a_1,a_2), (b_0,b_1,b_2)\in \calS_\Delta\), $( \sqrt{a_0 b_0}, \sqrt{a_1 b_1}, \sqrt{a_2 b_2}) \in \calS_\Delta$; see~\Cref{lem:Sd-log-convex}.
Then \Cref{lem:Sn-membership} is a consequence of the following \lcnamecref{lem:Sn-membership-separately} with \((d,\lambda,\mu) = (d_i,\lambda_i,\mu_i)\) for all \(1\le i\le \Delta\), together with convexity of \(\log\calS_\Delta\), which allows us to take coordinatewise geometric mean.

\begin{lemma}\label{lem:Sn-membership-separately}
    Let \(\Delta\ge2\) and \(1\le d\le \Delta\) be integers. Let \(\lambda,\mu\ge0\). Then
    \[
        \biggl(
        \frac{A_{d}(\lambda,\mu)^{\frac{\Delta}{d}}}{A_{d+1}(\lambda,\mu)^{\frac{\Delta}{d+1}}},\,
        \frac{B_{d}(\mu)^{\frac{\Delta}{d}}}{A_{d+1}(\lambda,\mu)^{\frac{\Delta}{d+1}}},\,
        \frac{B_{d}(\lambda)^{\frac{\Delta}{d}}}{A_{d+1}(\lambda,\mu)^{\frac{\Delta}{d+1}}}
        \biggr) \in \calS_\Delta.
    \]
\end{lemma}

To illustrate what~\Cref{lem:Sn-membership-separately} means, it is worth noting that the case $d=\Delta$ amounts to calculating the tangent plane of $z=A_{d+1}(x,y)^{\frac{1}{d+1}}$ at $(x,y)=(\lambda,\mu)$. Indeed,
\[
    \partial_x \Bigl(A_{d+1}^{\frac{1}{d+1}}\Bigr) = \frac{1}{d+1}A_{d+1}^{-\frac{d}{d+1}}\partial_x A_{d+1} = A_{d+1}^{-\frac{d}{d+1}}B_d(y)
    \quad\text{and}\quad
    \partial_y\Bigl(A_{d+1}^{\frac{1}{d+1}}\Bigr)=A_{d+1}^{-\frac{d}{d+1}}B_d(x),
\]
so the tangent plane at $(x,y)=(\lambda,\mu)$ is
\begin{align*}\label{eq:tangent_plane}
   z
   &= A_{d+1}(\lambda,\mu)^{-\frac{d}{d+1}}\big(B_d(\mu)(x-\lambda) + B_d(\lambda)(y-\mu)\big) + A_{d+1}(\lambda,\mu)^{\frac{1}{d+1}} \\
   &= A_{d+1}(\lambda,\mu)^{-\frac{d}{d+1}}\big(B_d(\mu)(x-\lambda) + B_d(\lambda)(y-\mu) + A_{d+1}(\lambda,\mu)\big)\\
   &=  A_{d+1}(\lambda,\mu)^{-\frac{d}{d+1}} \big(B_d(\mu)x + B_d(\lambda)y +A_{d}(\lambda,\mu)\big). \numberthis
\end{align*}
As $A_{d+1}(x,y)^{\frac{1}{d+1}}$ is concave, this plane always lies above $z=A_{d+1}(x,y)^{\frac{1}{d+1}}$, which proves the case $d=\Delta$.

Having seen this,~\Cref{lem:Sn-membership-separately} essentially states that $(a_0,a_1,a_2)\in \calS_d$ implies $(a_0^{\Delta/d},a_1^{\Delta/d},a_2^{\Delta/d})\in \calS_{\Delta}$, which is arguably the most technically challenging part. Before getting into it, let us give a proof that \(\log \calS_d\) is indeed convex. 

\begin{lemma}\label{lem:Sd-log-convex}
    For \(d\ge2\), \(\log\calS_d\) is convex.
\end{lemma}
\begin{proof}
    Let $h(w,x,y)$ be the `homogenisation' of $A_{d+1}$, i.e.,
    \[
        h(w,x,y)\defeq (d+1)d w^{d-1}xy + (d+1) w^d(x+y) + w^{d+1}.
    \]
    For $w>0$, $h(w,x,y)= w^{d+1} A_{d+1}(x/w,y/w)$.
    Then \((a_0,a_1,a_2)\in\calS_d\) if and only if
    \[
        a_0 w + a_1 x + a_2 y \ge h(w,x,y)^{\frac{1}{d+1}}\quad \forall w,x,y\ge0.
    \]
    Indeed, the `if' part follows from \(h(1,x,y) = A_{d+1}(x,y)\); to see the `only if' part, note that for \(w>0\),
    \[
    	a_0w + a_1x + a_2y
    	= w (a_0 + a_1(x/w) + a_2(y/w))
    	\ge w A_{d+1}(x/w,y/w)^{\frac{1}{d+1}}
    	= h(w,x,y)^{\frac{1}{d+1}},
    \]
    and the case \(w=0\) is clear since \(h(0,x,y)=0\).

    We show that for all \((a_0,a_1,a_2),(b_0,b_1,b_2)\in \calS_d\), \((\sqrt{a_0b_0}, \sqrt{a_1b_1}, \sqrt{a_2b_2})\) is in \(\calS_d\), using this description of $\calS_d$ in terms of $h$.
    Let $(w_1,x_1,y_1)\defeq(\sqrt{b_0/a_0}w,\sqrt{b_1/a_1}x, \sqrt{b_2/a_2}y) $. Then
    \begin{align*}
        \sqrt{a_0b_0} w + \sqrt{a_1b_1} x + \sqrt{a_2b_2} y
        = a_0 w_1 + a_1x_1 + a_2 y_1 \ge h (w_1,x_1,y_1)^{\frac{1}{d+1}}.
    \end{align*}
    By swapping the roles of the $a_i$ and $b_i$, let $(w_2,x_2,y_2)\defeq(\sqrt{a_0/b_0}w,\sqrt{a_1/b_1}x, \sqrt{a_2/b_2}y)$. Then
    \begin{align*}
        \sqrt{a_0b_0} w + \sqrt{a_1b_1} x + \sqrt{a_2b_2} y= b_0 w_2 + b_1 x_2 + b_2 y_2
        \ge h(w_2,x_2,y_2)^{\frac{1}{d+1}}.
    \end{align*}
    It then remains to show that
    \[
        h(w_1,x_1,y_1)\, h(w_2,x_2,y_2)
        \ge h( \sqrt{w_1w_2},\, \sqrt{x_1x_2},\, \sqrt{y_1y_2} )^2.
    \]
    Let \(\mathbf{u}_i\defeq (w_i,x_i,y_i)\) for \(i=1,2\) and \(\mathbf{v}\defeq ( \sqrt{w_1w_2}, \sqrt{x_1x_2}, \sqrt{y_1y_2} )\).
    Write $h(w,x,y)=\sum  c_{\alpha,\beta,\gamma}w^{\alpha}x^{\beta}y^{\gamma}$, i.e., the monomial representation. As $c_{\alpha,\beta,\gamma}> 0$ always, the Cauchy--Schwarz inequality gives
    \begin{align*}
        h(\mathbf{v}) &= \sum c_{\alpha,\beta,\gamma}(w_1w_2)^{\alpha/2}(x_1x_2)^{\beta/2}(y_1y_2)^{\gamma/2} = \sum \sqrt{c_{\alpha,\beta,\gamma}w_1^\alpha x_1^\beta y_1^\gamma} \sqrt{c_{\alpha,\beta,\gamma}w_2^\alpha x_2^\beta y_2^\gamma}
      \\  &\leq \Bigl(\sum c_{\alpha,\beta,\gamma}w_1^\alpha x_1^\beta y_1^\gamma\Bigr)^{\!1/2}
      \Bigl(\sum c_{\alpha,\beta,\gamma}w_2^\alpha x_2^\beta y_2^\gamma\Bigr)^{\!1/2}
      = h(\mathbf{u}_1)^{1/2}h(\mathbf{u}_2)^{1/2}.\qedhere
    \end{align*}
\end{proof}

\subsection{Proof of \texorpdfstring{\Cref{lem:Sn-membership-separately}}{Lemma~\ref{lem:Sn-membership-separately}}}
To sketch our proof of \Cref{lem:Sn-membership-separately}, first recall that \Cref{lem:Ad+1-concave} showed that \(A_{d+1}(x,y)^{\frac{1}{d+1}}\) is concave on \((\RR_{\ge0})^2\), so its tangent plane always lies above the surface \(z=A_{d+1}(x,y)^{\frac{1}{d+1}}\). In other words, the boundary of $\calS_d$ consists of those $(a_0,a_1,a_2)\in (\RR_{\geq 0})^3$ such that \(z=a_0+a_1x+a_2y\) is tangent to \(z=A_{d+1}(x,y)^{\frac{1}{d+1}}\) at a point in the nonnegative quadrant. Thus, to show whether a point $(v_0,v_1,v_2)$ is in $\calS_\Delta$, it is convenient to compute (or at least estimate) the value $v_0^*$, where $(v_0^*,v_1,v_2)$ lies on the boundary of $\calS_\Delta$, and check if \(v_0\ge v_0^*\).
It turns out that $v_0^*$ essentially depends on the product $v_1v_2$; see \Cref{lem:Phi-upper-bound} for the precise statement. This allows us to reduce \Cref{lem:Sn-membership-separately} to the case when \(\lambda=\mu\).
After the reduction, proving a still technical but low-dimensional inequality concludes the proof.

For \(a_1,a_2>0\), let
\begin{equation}\label{eq:min-a0-when-a1-a2-given}
    \Phi_\Delta(a_1,a_2)\defeq \inf \{a_0\ge0 : (a_0,a_1,a_2)\in \calS_\Delta\}
    = \sup_{x,y\ge0} \bigl( A_{\Delta+1}(x,y)^{\frac{1}{\Delta+1}} - a_1x - a_2y \bigr).
\end{equation}
Then for a given \((w_0,w_1,w_2)\in (\RR_{>0})^3\), \((w_0,w_1,w_2)\in\calS_\Delta\) if and only if \(w_0\ge \Phi_\Delta(w_1,w_2)\). The following \lcnamecref{lem:Phi-upper-bound} provides an upper bound for \(\Phi_\Delta(a_1,a_2)\) which only depends on \(a_1a_2\) up to a simple term.

\begin{lemma}\label{lem:Phi-upper-bound}
    There is a function \(\Psi_\Delta\colon(0,1)\to\RR\) such that for all \(a_1,a_2>0\) with \(a_1a_2<1\), 
    \[
        \Phi_\Delta(a_1,a_2) \le \Psi_\Delta(a_1a_2) + (a_1+a_2)/\Delta.
    \]
    Moreover, if \(a_1=a_2\), then equality holds.
\end{lemma}
\begin{proof}
    We want to identify the point \((x_*,y_*)\in\RR^2\) at which
    \[
        \partial_x \bigl( A_{\Delta+1}(x,y)^{\frac{1}{\Delta+1}} \bigr) = a_1
        \quad\text{and}\quad
        \partial_y \bigl( A_{\Delta+1}(x,y)^{\frac{1}{\Delta+1}} \bigr) = a_2.
    \]
    Denote by $\xi\defeq A_{\Delta+1}(x_*,y_*)^{\frac{1}{\Delta+1}}$.
    Since
    \begin{align*}
        \partial_x \bigl( A_{\Delta+1}(x,y)^{\frac{1}{\Delta+1}} \bigr)
        &= (\Delta y+1) \, A_{\Delta+1}(x,y)^{-\Delta/(\Delta+1)}
        \quad\text{and}\\
        \partial_y \bigl( A_{\Delta+1}(x,y)^{\frac{1}{\Delta+1}} \bigr)
        &= (\Delta x+1) \, A_{\Delta+1}(x,y)^{-\Delta/(\Delta+1)},
    \end{align*}
    $a_1=(1+\Delta y_*)\xi^{-\Delta}$ and $a_2=(1+\Delta x_*)\xi^{-\Delta}$. Using the identity \((\Delta+1)(\Delta x+1)(\Delta y+1) = \Delta A_{\Delta+1}(x,y) + 1\),
    $\xi$ is a zero of the polynomial $f(X)\defeq(\Delta+1)a_1a_2 X^{2\Delta} - \Delta X^{\Delta+1} - 1$.
    The following claim, together with the assumption \(0<a_1a_2<1\), guarantees the existence of a unique zero.
    
    \begin{claim}
        For \(0<s<1\), $f(X)$ has a unique zero \(\xi_\Delta (s)\) on \((1,\infty)\).
    \end{claim}
    \begin{claimproof}
        Note that \(f(0)=-1<0\), \(f(1)=(\Delta+1)(s-1)<0\), and \(f(X)\to\infty\) as \(X\to\infty\).
        In addition, \(f'(X)=\Delta(\Delta+1) X^\Delta (2sX^{\Delta-1}-1)\) has at most one positive zero \(\alpha\defeq (2s)^{-\frac{1}{\Delta-1}}\). Also, \(f'(X)<0\) on \((0,\alpha)\) and \(f'(X)>0\) on \((\alpha,\infty)\). Thus, \(f\) has a unique zero on \((0,\infty)\), which must be larger than \(1\).
    \end{claimproof}
    To emphasise the dependence on $a_1a_2$, write
    \( A_{\Delta+1}(x_*,y_*)^{\frac{1}{\Delta+1}}=\xi_\Delta(a_1a_2) > 1\). This determines
    \begin{align}\label{eq:x*y*}
        x_* = \frac{a_2 \xi_\Delta(a_1a_2)^\Delta - 1}{\Delta}
        \quad\text{and}\quad
        y_* = \frac{a_1 \xi_\Delta(a_1a_2)^\Delta - 1}{\Delta}.
    \end{align}
    In general, this \((x_*,y_*)\) need not lie in \((\RR_{\ge0})^2\). Instead, we consider a convex set containing \((\RR_{\ge0})^2\) and \((x_*,y_*)\); we claim \(\Omega\defeq \{(x,y)\in\RR^2 : A_{\Delta+1}(x,y)>0,\, y>-1/\Delta\}\) is suitable. That \((\RR_{\ge0})^2\subseteq \Omega\) is clear, and \(A_{\Delta+1}(x_*,y_*) = \xi_\Delta(a_1a_2)^{\Delta+1}>1\). In addition, \(y_*>-1/\Delta\) follows from \(\xi_\Delta(a_1a_2)>1\) and \(a_1>0\). To see that \(\Omega\) is convex, let
    \[
        g(y)\defeq -\frac{(\Delta+1)y+1}{(\Delta+1)(\Delta y+1)}
        \quad\text{so that}\quad
        \Omega = \{ (x,y)\in\RR^2 : x>g(y),\, y>-1/\Delta \}.
    \]
    Then \(g''(y)=2\Delta(\Delta+1)^{-1}(\Delta y+1)^{-3}>0\) on \(y>-1/\Delta\), so \(g\) is convex on \((-1/\Delta,\infty)\). Hence, \(\Omega\) is convex. Now by \Cref{lem:Ad+1-concave}, \(A_{\Delta+1}(x,y)^{\frac{1}{\Delta+1}}\) is concave on \(\Omega\), so
    \begin{align*}
        \Phi_\Delta(a_1,a_2)
        &\le \sup_{(x,y)\in\Omega} \bigl( A_{\Delta+1}(x,y)^{\frac{1}{\Delta+1}} - a_1x - a_2y \bigr)
        = A_{\Delta+1}(x_*,y_*)^{\frac{1}{\Delta+1}} - a_1x_* - a_2y_*
        \\&= \xi_\Delta(a_1a_2) - (2/\Delta)a_1a_2 \xi_\Delta(a_1a_2)^\Delta + (a_1+a_2)/\Delta.
    \end{align*}
    Let \(\Psi_\Delta(s)\defeq \xi_\Delta(s) - (2/\Delta) s\xi_\Delta(s)^\Delta\) for \(0<s<1\) so that $\Phi_\Delta(a_1,a_2)\leq \Psi_\Delta(a_1a_2) + (a_1+a_2)/\Delta$.

    To see the `moreover' part, suppose \(a_1=a_2=\sqrt{s}\) for some \(s>0\). Then by definition of $\xi_\Delta$,
    \[
        (\Delta+1) s \xi_\Delta(s)^{2\Delta} = \Delta \xi_\Delta(s)^{\Delta+1} + 1 > \Delta+1,
    \]
    so \(a_1 \xi_\Delta(a_1a_2)^{\Delta} = a_2 \xi_\Delta(a_1a_2)^{\Delta} > 1\). Then \(x_*=y_*>0\) by~\eqref{eq:x*y*} and therefore,
    \[
        \Phi_\Delta(a_1,a_2) = A_{\Delta+1}(x_*,y_*)^{\frac{1}{\Delta+1}} - a_1x_* - a_2y_* = \Psi_\Delta(a_1a_2) + (a_1+a_2)/\Delta. \qedhere
    \]
\end{proof}

The following inequality will allow us to check that the condition $a_1a_2<1$ in~\Cref{lem:Phi-upper-bound} holds in the proof of~\Cref{lem:Sn-membership-separately}.

\begin{proposition}\label{prop:basic-ineq}
	Let \(d\ge1\) and \(x,y\ge0\). Then
	\(
		(dx+1)^{d+1} (dy+1)^{d+1} \le \bigl( (d+1)dxy + (d+1)(x+y) + 1 \bigr)^{2d}
	\).
	Equality holds if and only if \(x=y=0\).
\end{proposition}
\begin{proof}
    Let \(t\defeq x+y+dxy\). Then the desired inequality is equivalent to \((dt+1)^{d+1} \le ((d+1)t+1)^{2d}\).
    Let \(f(t) \defeq 2d\log((d+1)t+1) - (d+1) \log(dt+1)\).
    Then \(f(0) = 0\) and for \(t\ge0\),
    \[
        f'(t) = d(d+1) \frac{(d-1)t+1}{(dt+1)((d+1)t+1)} > 0. \qedhere
    \]
\end{proof}

Now we are ready to prove \Cref{lem:Sn-membership-separately}.

\begin{proof}[Proof of~\Cref{lem:Sn-membership-separately}]
For brevity, let
\[
    A\defeq A_d(\lambda,\mu),\qquad
    B\defeq B_d(\mu),\qquad
    C\defeq B_d(\lambda),\qquad
    D\defeq A_{d+1}(\lambda,\mu),
\]
and
\begin{equation}\label{eq:wi-written-as-ABCD}
    w_0 \defeq A^{\frac{\Delta}{d}} D^{-\frac{\Delta}{d+1}},
    \qquad
    w_1 \defeq B^{\frac{\Delta}{d}} D^{-\frac{\Delta}{d+1}},
    \qquad
    w_2 \defeq C^{\frac{\Delta}{d}} D^{-\frac{\Delta}{d+1}}.
\end{equation}
Then the goal is to show that
\[
    (w_0,w_1,w_2)\in \calS_\Delta.
\]
Note that
\begin{equation}\label{eq:AD-written-as-BC}
    A=A(B,C)\defeq \frac{(d-1)BC+B+C-1}{d}
    \quad\text{and}\quad
    D=\frac{(d+1)BC-1}{d}.
\end{equation}

We have \(w_1w_2\le 1\) since \(B^{d+1}C^{d+1}\le D^{2d}\) by \Cref{prop:basic-ineq}. If \(w_1w_2=1\), then equality holds in \Cref{prop:basic-ineq}, implying \(\lambda=\mu=0\). Then the goal reduces to checking \((1,1,1)\in\calS_\Delta\). This indeed holds since for all \(x,y\ge0\),
\begin{align*}
    (1+x+y)^{\Delta+1}
    &\ge \binom{\Delta+1}{2}(x+y)^2 + (\Delta+1)(x+y) + 1
    \ge (\Delta+1)\Delta xy + (\Delta+1)(x+y) + 1
    \\&= A_{\Delta+1}(x,y).
\end{align*}

\step{Reduction to the symmetric case}
In what follows, assume \(w_1w_2<1\). Our first step is to reduce it to the case \(\lambda=\mu\).
Recall that \((w_0,w_1,w_2)\in\calS_\Delta\) is equivalent to \(w_0\ge \Phi_\Delta(w_1,w_2)\). By \Cref{lem:Phi-upper-bound},
\[
    w_0 - \Phi_\Delta(w_1,w_2)
    \ge w_0 - \frac{w_1+w_2}{\Delta} - \Psi_\Delta(w_1w_2).
\]
Consider \(w_0\), \(w_1\), and \(w_2\) as functions of \(B\) and \(C\), which is possible due to \eqref{eq:wi-written-as-ABCD} and \eqref{eq:AD-written-as-BC}. Note that \(w_1w_2\) is constant when \(BC\) is fixed, and hence, so is \(\Psi_\Delta(w_1w_2)\). We claim that if \(BC\) is fixed, then \(w_0 - (w_1+w_2)/\Delta\) is minimized when \(B=C\). Let \(K\defeq (BC)^{1/2}\) and write \(B=B(t)\defeq K e^t\) and \(C=C(t)\defeq K e^{-t}\). Without loss of generality, assume \(B\ge C\) so that \(t\ge0\). Let \(p\defeq \Delta/d\). Then
\[
    \frac{d}{dt} A(B,C)= \frac{B'(t) + C'(t)}{d} = \frac{B-C}{d}
    \quad\text{and}\quad
    \frac{d}{dt} (B^{p} + C^{p}) 	= p (B^p - C^p).
\]
By~\eqref{eq:AD-written-as-BC}, $D$ is fixed and thus, by~\eqref{eq:wi-written-as-ABCD},
\[
	\frac{d}{dt} (w_0-(w_1+w_2)/\Delta) = d^{-1} \bigl( p (B-C) A^{p-1} - (B^p - C^p) \bigr) D^{-\frac{\Delta}{d+1}}.
\]
This is nonnegative, since \(A\ge B\ge C\), where \(A\ge B\) follows from \(C\ge1\), so that
\[
	p(B-C)A^{p-1}=\int_C^B pA^{p-1} \diff\zeta
	\geq \int_C^B p\zeta^{p-1} \diff\zeta = B^p-C^p.
\]
Thus, \(w_0-(w_1+w_2)/\Delta\) is minimized when \(t=0\). Therefore,
\begin{align*}
    w_0(B,C) - \Phi_\Delta(w_1(B,C), w_2(B,C))
    &\ge w_0(B,C) - \frac{w_1(B,C)+w_2(B,C)}{\Delta} - \Psi_\Delta(w_1(B,C)w_2(B,C))
    \\&\ge w_0(K,K) - \frac{w_1(K,K)+w_2(K,K)}{\Delta} - \Psi_\Delta(w_1(B,C)w_2(B,C))
    \\&= w_0(K,K) - \Phi_\Delta(w_1(K,K), w_2(K,K)).
\end{align*}
Here the equality follows from the `moreover' part of \Cref{lem:Phi-upper-bound} with \(w_1(K,K) = w_2(K,K)\). In addition, \(w_1(B,C)w_2(B,C) = w_1(K,K)w_2(K,K)\) since \(w_1w_2\) depends only on \(BC\) and \(BC = K^2\). Hence, it suffices to consider the case \(B=C\), i.e., \(\lambda=\mu\).

\step{Analysis of the symmetric case}
Let \(A_d(\lambda) \defeq A_d(\lambda,\lambda) = d(d-1)\lambda^2 + 2d\lambda + 1\).
Recall that the goal \((w_0,w_1,w_2)\in\calS_\Delta\) with \(\lambda=\mu\) states that for all integer \(1\le d\le \Delta\) and reals \(\lambda\ge0\) and \(x,y\ge0\),
\[
    \bigl( A_d(\lambda)^{\frac{\Delta}{d}} + (x+y) B_d(\lambda)^{\frac{\Delta}{d}} \bigr) A_{d+1}(\lambda)^{-\frac{\Delta}{d+1}}
    \ge A_{\Delta+1}(x,y)^{\frac{1}{\Delta+1}}.
\]
As was shown in~\eqref{eq:tangent_plane}, if $d=\Delta$, then the left-hand side corresponds to the tangent plane of $z=A_{\Delta+1}(x,y)^{\frac{1}{\Delta+1}}$ at $(x,y)=(\lambda,\lambda)$ lying above the right-hand side, so the inequality holds.
Furthermore, if \(x+y\) is fixed, then \(A_{\Delta+1}(x,y)\), an increasing function of $xy$, is maximized when \(x=y\). Assuming \(x=y\), it suffices to show that for any $x\geq 0$, $\lambda\geq 0$, and $1\leq d\leq \Delta$, 
\[
    \bigl( A_d(\lambda)^{\frac{\Delta}{d}} + 2 B_d(\lambda)^{\frac{\Delta}{d}} \cdot x \bigr) A_{d+1}(\lambda)^{-\frac{\Delta}{d+1}}
    \ge A_{\Delta+1}(x)^{\frac{1}{\Delta+1}},
\]
which is shown to be true for $d=\Delta$ and any $x,\lambda\geq 0$. Hence, if for each $\lambda\geq 0$, there exists $\alpha\geq 0$ such that

\[
     \bigl( A_d(\lambda)^{\frac{\Delta}{d}} + 2 B_d(\lambda)^{\frac{\Delta}{d}} \cdot x \bigr) A_{d+1}(\lambda)^{-\frac{\Delta}{d+1}}
    \geq \big(A_\Delta(\alpha) + 2 B_\Delta(\alpha)\cdot x \big)A_{\Delta+1}(\alpha)^{-\frac{\Delta}{\Delta+1}}
\]
for all $x\geq 0$, then we are done.

To prove $ax+b\geq a'x+b'$ for all \(x\ge0\), it is enough to show $a\geq a'$ and $b\geq b'$. In our case, this means that it is enough to find \(\alpha\ge0\) such that
\begin{align*}
    A_d(\lambda)^{\frac{\Delta}{d}} A_{d+1}(\lambda)^{-\frac{\Delta}{d+1}}\ge A_\Delta(\alpha) A_{\Delta+1}(\alpha)^{-\frac{\Delta}{\Delta+1}}
    \quad\text{and}\quad 
    B_d(\lambda)^{\frac{\Delta}{d}} A_{d+1}(\lambda)^{-\frac{\Delta}{d+1}}
    \geq B_\Delta(\alpha) A_{\Delta+1}(\alpha)^{-\frac{\Delta}{\Delta+1}}.
\end{align*}
Rearranging these gives
\[
    \frac{A_\Delta(\alpha)^{\frac{1}{\Delta}}}{A_{\Delta+1}(\alpha)^{\frac{1}{\Delta+1}}} \le \frac{A_d(\lambda)^{\frac{1}{d}}}{A_{d+1}(\lambda)^{\frac{1}{d+1}}}
    \quad\text{and}\quad
    \frac{B_\Delta(\alpha)^{\frac{1}{\Delta}}}{A_{\Delta+1}(\alpha)^{\frac{1}{\Delta+1}}} \le \frac{B_d(\lambda)^{\frac{1}{d}}}{A_{d+1}(\lambda)^{\frac{1}{d+1}}}.
\]
It suffices to prove the adjacent case \(\Delta=d+1\), since iterating this comparison gives the result for arbitrary \(d\le\Delta\). Thus, we want to show that there exists $\alpha=\alpha(\lambda,d)\geq 0$ such that
\begin{equation}\label{eq:alpha-lambda-conditions}
    \frac{A_{d+1}(\alpha)^{\frac{1}{d+1}}}{A_{d+2}(\alpha)^{\frac{1}{d+2}}} \le \frac{A_d(\lambda)^{\frac{1}{d}}}{A_{d+1}(\lambda)^{\frac{1}{d+1}}}
    \quad\text{and}\quad
    \frac{B_{d+1}(\alpha)^{\frac{1}{d+1}}}{A_{d+2}(\alpha)^{\frac{1}{d+2}}} \le \frac{B_d(\lambda)^{\frac{1}{d}}}{A_{d+1}(\lambda)^{\frac{1}{d+1}}}.
\end{equation}
For \(k\in \{d,d+1\}\), let \(H_k(x)\defeq A_k(x)/B_k(x)\).
Then by recalling \(A_k(x) = k(k-1)x^2 + 2kx + 1\) and \(B_k(x) = kx+1\),
\begin{equation}\label{eq:Hk-property}
    H_k(0)=1
    \quad\text{and}\quad
    H_k'(x) = \frac{k( (k-1)(kx^2+2x) + 1)}{B_k(x)^2} > 0,
\end{equation}
so as a function of \(x\in [0,\infty)\), $H_k(x)^{1/k}$ is strictly increasing and $H_k(x)^{1/k}\geq 1$. Moreover,
\[
    H_k(x)\to \begin{cases*}
        2 & if \(k=1\), \\
        \infty & if \(k\ge2\),
    \end{cases*}
    \quad\text{as }x\to\infty.
\]
Thus, if \(k\ge2\), for each \(s\ge1\), there is a unique \(x_k=x_k(s)\ge0\) such that \(H_k(x_k)^{1/k}=s\). For $k=1$, a unique $x_1=x_1(s)$ exists such that \(H_1(x_1)=s\) whenever $1\leq s<2$.
Let \(s\defeq H_{d}(\lambda)^{\frac{1}{d}}\ge 1\) and \(\alpha\defeq x_{d+1}(s)\). Then \(\lambda=x_{d}(s)\) by definition and \(H_{d+1}(\alpha)^{\frac{1}{d+1}} = s= H_{d}(\lambda)^{\frac{1}{d}}\).
Note that, for $d=1$, $x_1(s)$ is well-defined as $1\leq H_1(\lambda)<2$ always.
If the second condition in \eqref{eq:alpha-lambda-conditions} holds for this particular choice of $\alpha=\alpha(\lambda,d)$, i.e.,
\begin{equation}\label{eq:goal-reduced-1}
    \frac{B_{d+1}(x_{d+1}(s))^{\frac{1}{d+1}}}{A_{d+2}(x_{d+1}(s))^{\frac{1}{d+2}}}=
    \frac{B_{d+1}(\alpha)^{\frac{1}{d+1}}}{A_{d+2}(\alpha)^{\frac{1}{d+2}}} \le \frac{B_d(\lambda)^{\frac{1}{d}}}{A_{d+1}(\lambda)^{\frac{1}{d+1}}}=\frac{B_d(x_d(s))^{\frac{1}{d}}}{A_{d+1}(x_d(s))^{\frac{1}{d+1}}},
\end{equation}
then it implies the first condition, as
\[
    \frac{A_{d+1}(\alpha)^{\frac{1}{d+1}}}{A_{d+2}(\alpha)^{\frac{1}{d+2}}}
    = \frac{B_{d+1}(\alpha)^{\frac{1}{d+1}}}{A_{d+2}(\alpha)^{\frac{1}{d+2}}}
        \cdot H_{d+1}(\alpha)^{\frac{1}{d+1}}
    =\frac{B_{d+1}(\alpha)^{\frac{1}{d+1}}}{A_{d+2}(\alpha)^{\frac{1}{d+2}}}
        \cdot H_{d}(\lambda)^{\frac{1}{d}}
    \le \frac{B_d(\lambda)^{\frac{1}{d}}}{A_{d+1}(\lambda)^{\frac{1}{d+1}}}
        \cdot H_{d}(\lambda)^{\frac{1}{d}}
    = \frac{A_d(\lambda)^{\frac{1}{d}}}{A_{d+1}(\lambda)^{\frac{1}{d+1}}}.
\]

For \(k\in\{d,d+1\}\), let
\[
    R_k(s) \defeq \frac{B_k(x_k(s))^{\frac{1}{k}}}{A_{k+1}(x_k(s))^{\frac{1}{k+1}}}.
\]
To prove~\eqref{eq:goal-reduced-1}, it is enough to show \(R_{d+1}(s)\le  R_{d}(s)\) for all \(s\ge1\). Technically, one should assume $1\leq s<2$ for the case $d=1$ to maintain the well-definedness of $x_1(s)$, but this makes no changes in the arguments.
To this end, we compute $\partial_s \log R_{k}(s)$.
Writing \(x\defeq x_k(s)\) and using \(A_k(x)/B_k(x) = s^k\) and \(A_{k+1}(x) = A_k(x) + 2x B_k(x)=(s^k + 2x)B_k(x)\),
\begin{align*}
    \partial_s \log R_k(s)
    &= \left(\frac{1}{k} \frac{B_k'(x)}{B_k(x)} - \frac{1}{k+1} \frac{A_{k+1}'(x)}{A_{k+1}(x)}\right)\frac{dx}{ds}
    = \left(  \frac{1}{B_k(x)}  - \frac{2B_k(x)}{A_{k+1}(x)} \right)\frac{dx}{ds}
    \\&= \frac{s^k - 2(k-1)x - 2}{(s^k+2x) B_k(x)}\frac{dx}{ds}.
    \numberthis\label{eq:log-Phik-diff-calc-1}
\end{align*}
Differentiating both sides of \(A_k(x) = s^k B_k(x)\) by $s$ gives $A_k'(x)\frac{dx}{ds}= ks^{k-1} B_k(x) + s^k B_k'(x)\frac{dx}{ds}$, so
\[
    \frac{dx}{ds} = \frac{ks^{k-1} B_k(x)}{A_k'(x) - s^k B_k'(x)} = \frac{s^{k-1} B_k(x)}{2(k-1)x + 2 - s^k}.
\]
This together with \eqref{eq:log-Phik-diff-calc-1} gives 
\begin{equation}\label{eq:log-Phik-diff}
    \partial_s \log R_{k}(s)
    = -\frac{s^{k-1}}{s^{k}+2x_{k}(s)}.
\end{equation}
Suppose \(x_{d+1}(s) \le s x_{d}(s)\). Then \eqref{eq:log-Phik-diff} implies
\[
    \partial_s \log R_{d+1}(s) \le \partial_s \log R_{d}(s),
\]
and using \(R_{d+1}(1) = R_{d}(1) = 1\), we have \(R_{d+1}(s) \le R_{d}(s)\). Thus, it suffices to show \(x_{d+1}(s)\le s x_{d}(s)\).

\medskip

By \eqref{eq:Hk-property}, \(H_{d+1}\) is strictly increasing, so \(x_{d+1}(s)\le s x_{d}(s)\) is equivalent to
\[
    s^{d+1} = H_{d+1}(x_{d+1}(s)) \le H_{d+1}(s\, x_{d}(s)).
\]
Write \(x \defeq x_{d}(s)\). By definition,
\[
    H_{d+1}(sx) = \frac{A_{d+1}(sx)}{B_{d+1}(sx)} = 1 + dsx + \frac{sx}{(d+1)sx+1}.
\]
On the other hand,
\[
    s^{d+1} = s\cdot s^{d} = s H_{d}(x) = s \Bigl( 1 + (d-1)x + \frac{x}{dx+1} \Bigr).
\]
Recall that $s\geq 1$, which implies \(d - (d+1)s \le 0\). Using this, we have
\begin{align*}
    H_{d+1}(sx) - s^{d+1}
    &= (1-s) + sx + sx \Bigl( \frac{1}{(d+1)sx+1} - \frac{1}{dx+1} \Bigr)
    \\&= (1-s) + sx + sx \frac{(d - (d+1)s)x}{((d+1)sx+1)(dx+1)}
    \\&\ge (1-s) + sx + sx \frac{(d - (d+1)s)x}{((d+1)s + d)x + 1}.
    \numberthis\label{eq:Hk-diff}
\end{align*}
 Let \(Q(x)\) be the expression obtained by multiplying \eqref{eq:Hk-diff} by \(((d+1)s + d)x+1\). Then
\[
    Q(x) = (1-s)\big(((d+1)s +d)x +1\big) + sx(2dx+1) =
    2dsx^2 + ( 2s +d -s^2d -s^2 )x - (s-1).
\]
As $2ds>0$, \(Q(0)\le 0\), and \(Q(s-1)=(d-1)(s-1)^3\ge0\), this quadratic polynomial $Q(x)$ is nonnegative on \(x\ge s-1\). Indeed, $x=x_d(s)\geq s-1$, as
\[
    1 + d(s-1) \le s^{d} = H_{d}(x) = 1 + (d-1)x + \frac{x}{dx+1} \le 1 + dx.
\]
Hence, \(H_{d+1}(sx) \ge s^{d+1}\), which concludes the proof.
\end{proof}

\begin{remark}
    While \Aletheia's output was closer to an incomplete proof sketch than a complete proof, it contained several significant ideas. For instance, it suggested using the dual set $\mathcal{S}_d$ and its log-convexity; it also provided a preliminary form of \Cref{lem:Phi-upper-bound} and the reduction technique for \Cref{lem:Sn-membership-separately}. Based on these insights, we formalised the arguments and filled in the necessary details. In particular, for \Cref{lem:Sd-log-convex} and the symmetric case of \Cref{lem:Sn-membership-separately}, we retained only the statements provided by the model and wrote the proofs entirely independently, using different methods to ensure clarity and simplicity.
\end{remark}


\section{More towards~\texorpdfstring{\Cref{conj:main}}{Conjecture~\ref{conj:main}}}
Recall that \Cref{conj:main} states that for any graph \(G\) and antiferromagnetic edge-weighted graph \(H\),
\begin{equation}\label{eq:afm-copy}
    \hom(G,H)\ge \prod_{v\in V(G)} \hom(K_{d_v+1},H)^{\frac{1}{d_v+1}}.
\end{equation}
We show that this holds if \(\Delta(G)\le 2\).
\begin{theorem}\label{thm:conj-main-deg<=2}
    \Cref{conj:main} is true if \(\Delta(G)\le 2\).
\end{theorem}
\begin{proof}
    Since both sides of \eqref{eq:afm-copy} are multiplicative over the connected components of \(G\), it suffices to consider the case when \(G\) is connected. Thus, we assume \(G\) is either a path or a cycle.
    Let $\mu_1\geq \mu_2\geq\dots\geq \mu_n$ be the spectrum of $H$.

    Suppose first that $G$ is an $\ell$-edge path $P_\ell$ where \(\ell\ge1\).
    By a classical result of Cvetkovi\'{c}, Doob, and Sachs~\cite[Theorem~1.10]{cvetkovic1980spectra} (also see~\cite[Theorem~1.3.5]{cvetkovic2010introduction} for a recent source), there exist nonnegative reals $p_1,\dots,p_n$, $\sum_{i=1}^n p_i=n$, such that the (weighted) number of $m$-edge walks in $H$ is $p_1\mu_1^m+\cdots+p_n\mu_n^m$ for each $m$. We then have, by using the antiferromagnetism condition $\mu_i\leq 0$ and \(\abs{\mu_i}\le \mu_1\) for each $i>1$,
    \begin{align*}
        \hom(P_\ell,H)=p_1\mu_1^\ell +\dots +p_n\mu_n^\ell\geq \mu_1^{\ell-1}(p_1\mu_1+\cdots+p_n\mu_n).
    \end{align*}
    Again by antiferromagnetism, $\mu_1^3 \geq \mu_1^3 +\dots +\mu_n^3$, so
    \begin{align*}
        \mu_1^{\ell-1}(p_1\mu_1+\cdots+p_n\mu_n)
        &\geq (\mu_1^3+\cdots+\mu_n^3)^{\frac{\ell-1}{3}}(p_1\mu_1 +\cdots + p_n\mu_n)\\
        &=\hom(K_3,H)^{\frac{\ell-1}{3}}\hom(K_2,H),
    \end{align*}
    which is the desired bound.

    Suppose now that $G$ is a copy of $C_\ell$, the cycle of length $\ell\geq 3$. Then, by using $\abs{\mu_i}\le \mu_1$ for each $i$,
     \begin{align*}
        \hom(C_\ell,H)=\mu_1^\ell +\cdots + \mu_n^\ell \geq \mu_1^{\ell-3}(\mu_1^3 +\cdots + \mu_n^3) \geq (\mu_1^3 +\cdots +\mu_n^3)^{\frac{\ell}{3}} = \hom(K_3,H)^{\frac{\ell}{3}},
    \end{align*}
    where the last inequality uses $\mu_1^3\geq \mu_1^3 +\cdots +\mu_n^3$. This concludes the proof.
\end{proof}

\medskip

We say that an edge-weighted graph $H$ is a \emph{$q$-spin antiferromagnetic model} if it is antiferromagnetic and $\abs{V(H)}=q$. For instance, $H=\indepg$ and any graph obtained by assigning positive edge weights to the two edges of $H$ are \(2\)-spin antiferromagnetic models. \Cref{thm:2spin} hence shows that~\Cref{conj:main} holds for all 2-spin antiferromagnetic models of such kind.
Let \(K_{q+1}^\circ\) be the graph obtained from \(K_{q+1}\) by attaching a loop to one vertex, which is a $(q+1)$-spin antiferromagnetic model.
This is one of the most important $(q+1)$-spin antiferromagnetic models, as the support of an antiferromagnetic graph is always a blow-up of \(K_{q+1}\) or \(K_{q+1}^\circ\) after removing isolated vertices~\cite[Theorem~3.2]{lee2025counting}. 
Our goal is to obtain a slightly weaker bound than~\eqref{eq:afm-copy} for these graphs.

As was described in the introduction, when the host graph \(H\) is \(K_2^\circ = \indepg\), a homomorphism from \(G\) to \(H\) corresponds to an independent set of \(G\). When \(H=K_{q+1}^\circ\), a homomorphism from \(G\) to \(H\) corresponds to a \emph{semiproper colouring} of \(G\), a relaxation of proper vertex colourings by allowing \emph{looped} colours that can be used with no constraints. In particular, a homomorphism from \(G\) to \(K_{q+1}^\circ\) corresponds to a semiproper colouring of \(G\) with \(q\) proper colours and one `free' colour.

Let \(q\) be a positive integer.
For reals \(\lambda^{(i)}\), \(1\le i\le q\), define
\[
    Z^{(q)}_G(\lambda^{(1)}, \dots, \lambda^{(q)})
    \defeq \sum_{(I_1,\dots,I_q)\in \calI^{(q)}(G)} \prod_{i=1}^q (\lambda^{(i)})^{\abs{I_i}},
\]
where \(\calI^{(q)}(G)\) consists of all \(q\)-tuples \((I_1,\dots,I_q)\) of pairwise disjoint (possibly empty) independent sets of \(G\). In particular, \(Z^{(q)}_G(1,1, \dots, 1)=\hom(G,K_{q+1}^\circ)\).
The multiaffine variant of \(Z^{(q)}_G\) is defined as follows: for tuples of reals \(\blambda^{(i)} = (\lambda^{(i)}_v)_{v\in V(G)}\), \(1\le i\le q\),
\[
    Z^{(q)}_G(\blambda^{(1)}, \dots, \blambda^{(q)})
    \defeq \sum_{(I_1,\dots,I_q)\in \calI^{(q)}(G)} \prod_{i=1}^q \prod_{v\in I_i} \lambda_v^{(i)}.
\]
If \(q=1\), \(Z^{(1)}_G(\blambda^{(1)})\) is exactly the multivariate independence polynomial \(Z_G(\blambda^{(1)})\).

We conjecture that the following analogous statement of \Cref{thm:indep-poly-multiaff-lower-bd} holds for \(Z^{(q)}_G\). This is a vertex-weighted version of \Cref{conj:main} when \(H=K_{q+1}^\circ\).

\begin{conjecture}\label{conj:semiproper-colouring-poly-multiaff-lower-bd}
    Let \(G\) be a graph and \(\blambda^{(i)} = (\lambda^{(i)}_v)_{v\in V(G)}\), \(1\le i\le q\) be tuples of nonnegative reals. Then
    \[
        Z^{(q)}_G(\blambda^{(1)}, \dots, \blambda^{(q)})
        \ge \prod_{v\in V(G)} Z^{(q)}_{K_{d_v+1}} (\lambda^{(1)}_v,\dots,\lambda^{(q)}_v) ^{\frac{1}{d_v+1}}.
    \]
\end{conjecture}

We shall prove a weaker bound as an application of \Cref{thm:indep-poly-multiaff-lower-bd}, which matches the bound in the conjecture in an asymptotic sense.
For graphs \(G\) and \(K\), their \emph{Cartesian product} (or \emph{box product}) \(G\cart K\) is the graph on \(V(G)\times V(K)\) in which two vertices \((u_1,v_1)\) and \((u_2,v_2)\) are adjacent if and only if either
\begin{enumerate*}
    \item \(u_1=u_2\) and \(v_1v_2\in E(K)\) or
    \item \(u_1u_2\in E(G)\) and \(v_1=v_2\).
\end{enumerate*}
Let $K_t$ be a complete graph on \([t]\defeq \{1,\dots,t\}\) so that \(V(G\cart K_t) = V(G)\times [t]\).
Then the map
\[
    (I_1,I_2)\mapsto \{(v,1) : v\in I_1\} \sqcup \{(v,2) : v\in I_2\}.
\]
is a bijection from \(\calI^{(2)}(G)\) to \(\calI(G\cart K_2)\).
Let us illustrate how $Z^{(2)}_G(\blambda^{(1)}, \blambda^{(2)})$ corresponds to $Z_{G\cart K_2}$ with suitable parameters. 
For simplicity, let $\blambda^{(i)}=(\lambda^{(i)},\dots,\lambda^{(i)})$, $i=1,2$,  for \(\lambda^{(1)},\lambda^{(2)}\in\RR_{\ge0}\).
Then setting \(\bm{\mu}=(\mu_{(v,i)})_{(v,i)\in V(G\cart K_2)}\) where \(\mu_{(v,i)}\defeq \lambda^{(i)}\in \RR\) for \((v,i)\in V(G\cart K_2)\) gives
\begin{align*}
    Z^{(2)}_G(\blambda^{(1)}, \blambda^{(2)})
    &= \sum_{(I_1,I_2)\in \calI^{(2)}(G)} (\lambda^{(1)})^{\abs{I_1}} \, (\lambda^{(2)})^{\abs{I_2}}
    = \sum_{I\in \calI(G\cart K_2)} \prod_{(v,i)\in I} \mu_{(v,i)}
    \\&= Z_{G\cart K_2} (\bm{\mu}).
\end{align*}
By \Cref{thm:indep-poly-multiaff-lower-bd} together with the fact \(\deg_{G\cart K_2}(v,i)=d_v+1\) where \(d_v\defeq\deg_G(v)\),
\begin{align*}
    Z_{G\cart K_2} (\bm{\mu})
    &\ge \prod_{(v,i)\in V(G\cart K_2)} Z_{K_{d_v+2}}(\mu_{(v,i)})^{\frac{1}{d_v+2}}
    \\&= \prod_{v\in V(G)} (1+(d_v+2)\lambda^{(1)})^{\frac{1}{d_v+2}} \, (1+(d_v+2)\lambda^{(2)})^{\frac{1}{d_v+2}}.
\end{align*}
If \(\lambda^{(1)}, \lambda^{(2)} = o(1/\Delta(G))\), the factors in this lower bound for \(Z_G^{(2)}(\lambda^{(1)},\lambda^{(2)})\) match the conjectured factors of~\Cref{conj:semiproper-colouring-poly-multiaff-lower-bd} up to the linear terms. Indeed, as \(\Delta(G)\to \infty\),
\begin{align*}
    (1+(d_v+2)\lambda^{(1)})^{\frac{1}{d_v+2}} \, (1+(d_v+2)\lambda^{(2)})^{\frac{1}{d_v+2}}
    &= \bigl( 1+(d_v+2)(\lambda^{(1)} + \lambda^{(2)}) + (d_v+2)^2\lambda^{(1)}\lambda^{(2)} \bigr)^{\frac{1}{d_v+2}}
    \\&= 1 + (\lambda^{(1)}+\lambda^{(2)}) + o(\lambda^{(1)}+\lambda^{(2)}),
\end{align*}
while
\begin{align*}
    Z^{(2)}_{K_{d_v+1}}(\lambda^{(1)}, \lambda^{(2)})^{\frac{1}{d_v+1}}
    &= (1+(d_v+1)(\lambda^{(1)}+\lambda^{(2)}) + (d_v+1)d_v\lambda^{(1)}\lambda^{(2)})^{\frac{1}{d_v+1}}
    \\&= 1+(\lambda^{(1)}+\lambda^{(2)}) + o(\lambda^{(1)}+\lambda^{(2)}).
\end{align*}
If we assume further that \(\lambda^{(1)}=\lambda^{(2)}=\lambda=o(1/\Delta(G))\), then 
\newlength{\eqlen}
\settowidth{\eqlen}{\(Z^{(2)}_{K_{d_v+1}}(\lambda^{(1)}, \lambda^{(2)})^{\frac{1}{d_v+1}}\)}
\begin{align*}
    &\hspace{-\eqlen}
    (1+(d_v+2)\lambda^{(1)})^{\frac{1}{d_v+2}} \, (1+(d_v+2)\lambda^{(2)})^{\frac{1}{d_v+2}}
    \\&= \bigl( 1+2(d_v+2)\lambda + (d_v+2)^2\lambda^2 \bigr)^{\frac{1}{d_v+2}}
    \\&= 1 + (2\lambda + (d_v+2)\lambda^2) - \frac{1}{2}\frac{1}{d_v+2}\frac{d_v+1}{d_v+2}(2(d_v+2)\lambda)^2 + O(d_v^2\lambda^3)
    \\&= 1+ 2\lambda - d_v\lambda^2 + O(d_v^2\lambda^3),
    \intertext{while}
    Z^{(2)}_{K_{d_v+1}}(\lambda^{(1)}, \lambda^{(2)})^{\frac{1}{d_v+1}}
    &= (1+2(d_v+1)\lambda + (d_v+1)d_v\lambda^2)^{\frac{1}{d_v+1}}
    \\&= 1 + (2\lambda + d_v\lambda^2) - \frac{1}{2}\frac{1}{d_v+1}\frac{d_v}{d_v+1}(2(d_v+1)\lambda)^2 + O(d_v^2\lambda^3)
    \\&= 1 + 2\lambda - d_v\lambda^2 + O(d_v^2\lambda^3),
\end{align*}
so the factors match up to the quadratic terms.
This argument generalises to the following \lcnamecref{thm:semiproper-colouring-poly-multiaff-lower-bd-weak}, which can be seen as an approximate version of~\Cref{conj:semiproper-colouring-poly-multiaff-lower-bd}.

\begin{theorem}\label{thm:semiproper-colouring-poly-multiaff-lower-bd-weak}
    Let \(G\) be a graph, \(q\ge1\) be an integer, and \(\blambda^{(i)}=(\lambda^{(i)}_v)_{v\in V(G)}\), \(1\le i\le q\) be tuples of nonnegative reals. Then
    \[
        Z^{(q)}_G(\blambda^{(1)}, \dots, \blambda^{(q)})
        \ge \prod_{v\in V(G)} \prod_{i=1}^q Z_{K_{d_v+q}} (\lambda^{(i)}_v)^{1/(d_v+q)}.
    \]
    In particular, if \(\max_{v\in V(G),\, 1\le i\le q} \lambda^{(i)}_v=o(1/\Delta(G))\), as \(\Delta(G)\to \infty\),
    \[
        Z^{(q)}_G(\blambda^{(1)}, \dots, \blambda^{(q)})
        \ge \prod_{v\in V(G)} \biggl( Z^{(q)}_{K_{d_v+1}} (\lambda^{(1)}_v,\dots,\lambda^{(q)}_v) ^{\frac{1}{d_v+1}} - o\biggl(\sum_{i=1}^q \lambda^{(i)}_v\biggr) \biggr).
    \]
    If in addition, for each \(v\in V(G)\), \(\lambda^{(1)}_v=\dots=\lambda^{(q)}_v\eqdef \lambda_v\) for all \(1\le i\le q\),
    \[
        Z^{(q)}_G(\blambda^{(1)}, \dots, \blambda^{(q)})
        \ge \prod_{v\in V(G)} \Bigl( Z^{(q)}_{K_{d_v+1}} (\lambda_v,\dots,\lambda_v) ^{\frac{1}{d_v+1}} - O_q(d_v^2\lambda_v^3) \Bigr).
    \]
\end{theorem}

In proving~\Cref{thm:semiproper-colouring-poly-multiaff-lower-bd-weak}, the most technical part is to extend the aforementioned bijection from \(\calI^{(2)}(G)\) to \(\calI(G\cart K_2)\). Once we have the bijection, it is straightforward to deduce \Cref{thm:semiproper-colouring-poly-multiaff-lower-bd-weak}, which we omit.

\begin{lemma}\label{lem:multiaffine-equal}
    Let \(G\) be a graph, \(q\ge1\) be an integer, and \(\blambda^{(i)}=(\lambda^{(i)}_v)_{v\in V(G)}\), \(1\le i\le q\) be tuples of nonnegative reals. Let \(\blambda\defeq (\lambda^{(i)}_v)_{(v,i)\in V(G\cart K_q)}\). Then
    \[
        Z^{(q)}_G(\blambda^{(1)}, \dots, \blambda^{(q)})
        = Z_{G\cart K_q}(\blambda).
    \]
\end{lemma}

This may be seen as a vertex-weighted version of a particular case of \cite[Lemma~3.1]{kim2016sidorenko}, which is on \(\hom(G\cart K,H)\) for graphs \(G\), \(K\), and \(H\).
It is clear that \Cref{thm:indep-poly-multiaff-lower-bd,lem:multiaffine-equal} imply the first conclusion of \Cref{thm:semiproper-colouring-poly-multiaff-lower-bd-weak}, since each vertex \((v,i)\) in \(G\cart K_q\) has degree \(\deg_G(v)+q-1\).

\begin{proof}[Proof~of~\Cref{lem:multiaffine-equal}]
    For each independent set \(I\) of \(G\cart K_q\) and $1\le i\le q$, let
    \[
        f_i(I)\defeq \{v\in V(G) : (v,i)\in I\}.
    \]
    We claim that \(I\mapsto (f_1(I),\dots,f_q(I))\) is a bijection from \(\calI(G\cart K_q)\) to \(\calI^{(q)}(G)\).
    First, if \(I\in \calI(G\cart K_q)\), then \(I\cap (V(G)\times\{i\})\in \calI(G\cart K_q)\) for each \(i\), so \(f_i(I) = \{v\in V(G) : (v,i)\in I\cap (V(G)\times \{i\})\} \in \calI(G)\). Also, if \(f_i(I)\cap f_j(I)\) contains a vertex, say \(v\), for some \(i\neq j\), then the edge \((v,i)(v,j)\) lies inside \(I\), a contradiction. Thus, the independent sets \(f_i(I)\), \(1\le i\le q\) are mutually disjoint.
    
    On the other hand, suppose \((I_1,\dots,I_q)\in \calI^{(q)}(G)\). Let \(I\defeq \{(v,i) : v\in I_i, 1\le i\le q\}\subseteq V(G\cart K_q)\). Then \(f_i(I)=I_i\) for each \(i\). Also, \(I\in \calI(G\cart K_q)\); otherwise, \(I\) contains an edge of the form \((u,i)(v,i)\) or \((u,i)(u,j)\) where \(u\neq v\) and \(i\neq j\), where the former induces an edge \(uv\) lying inside \(I_i\), and the latter implies that \(I_i\cap I_j\) is nonempty.
    
    Using this bijection,
    \[
        Z^{(q)}_G(\blambda^{(1)},\dots,\blambda^{(q)})
        = \sum_{(I_1,\dots,I_q)\in \calI^{(q)}(G)} \prod_{i=1}^q \prod_{v\in I_i} \lambda^{(i)}_v
        = \sum_{I\in \calI(G\cart K_q)} \prod_{(v,i)\in I} \lambda^{(i)}_v
        = Z_{G\cart K_q}(\blambda).
        \qedhere
    \]
\end{proof}

\section{Concluding remarks}

\paragraph{\Cref{conj:semiproper-colouring-poly-multiaff-lower-bd} for $q>2$} To generalise~\Cref{thm:semiprop-multiaff-lower-bd} to prove~\Cref{conj:semiproper-colouring-poly-multiaff-lower-bd} for $q>2$, one should replace $A_d(x,y)$ by a $q$-variate function $F_d(x_1,x_2,\dots,x_q)\defeq \sum_{k=0}^{q} (d)_k e_{k}(\mathbf{x})$, where $(d)_k$ denotes the falling factorial $d(d-1)\cdots(d-k+1)$ and $e_k(\mathbf{x})$ is the elementary symmetric function of degree $k$. Following the inductive approach, it is natural to replace $B_d$ by the first-order derivatives of $F_d$.
Then $F_d^{1/d}$ is again concave, and the dual set $\calS_d$ and its log-convexity generalises with no problems. However, \Cref{lem:Sn-membership-separately}, which heavily relies on technical advantages in low dimensions including \Cref{lem:Phi-upper-bound}, seems to be the bottleneck. More precisely, as $q$ grows, it gets harder to compute $a_0$ such that $(a_0,a_1,\dots,a_q)$ is in the boundary of $\calS_d$. We believe that $q=4$ would already require substantially new ideas, whereas it might be possible to handle the case $q=3$ with analogous techniques to ours. 

\paragraph{Extensions to negative fugacity}  If $G$ is a (hyper)graph of maximum degree $\Delta$ and each $\lambda_v$, $v\in V(G)$, is a complex number with modulus at most $\frac{\Delta^\Delta}{(\Delta+1)^{\Delta+1}}$, \cite[Theorem~10]{GMPST_2024} obtains a lower bound of the form
\begin{align*}
    Z_G(\blambda)\geq \left(1-\frac{1}{\Delta+1}\right)^{\abs{V(G)}}.
\end{align*}
As the right-hand side is smaller than $1$, this is weaker than~\Cref{thm:indep-poly-multiaff-lower-bd} for nonnegative values of the $\lambda_v$.
Having seen this comparison, one might wonder whether the range of $\lambda_v$'s in~\Cref{thm:indep-poly-multiaff-lower-bd} extends further to negative, or even complex, values. Unfortunately,~\Cref{thm:indep-poly-multiaff-lower-bd} does not extend to such a range, which was also observed by Gemini 3.0 Deep Think.
To elaborate, it is worth noting that substituting the value of $\lambda_0$ that makes the partial derivative of $(A+\lambda_0)(1+(\Delta+1)\lambda_0)^{-\frac{1}{\Delta+1}}$ with respect to $\lambda_0$ vanish also gives the inequality~\eqref{eq:key-tech-ineq-reduced}. Thus,~\eqref{eq:goal} holds if and only if~\eqref{eq:key-tech-ineq-reduced} holds. But then, analysing the Taylor expansion of $(\Delta+1)A-\Delta B^{\frac{\Delta+1}{\Delta}}-1$ around $\lambda_v=0$ proves that~\eqref{eq:key-tech-ineq-reduced} does not hold locally for some small negative values of $\lambda_v$.

\paragraph{Occupancy fractions} When proving inequalities on $Z_G(\lambda)$, or more generally on $\hom(G,H)$, the \emph{occupancy fraction method} is to compare  $\frac{d}{d\lambda}\log Z_G(\lambda)$ with $\frac{d}{d\lambda}\log L(\lambda)$, where $L(\lambda)$ is a candidate for the upper/lower bound for $Z_G(\lambda)$.
Here $\alpha_G(\lambda)\defeq \frac{\lambda}{\abs{V(G)}} \frac{d}{d\lambda}\log Z_G(\lambda)$, the so-called \emph{occupancy fraction}, has a clear combinatorial interpretation; it corresponds to the expected size of a randomly sampled independent set according to the hard-core model. The occupancy fraction method has seen a lot of success in proving combinatorial inequalities, e.g.,~\cite{DJPR_2017,davies2018average}.
One of the interesting conjectures in the area is due to Davies and Kang~\cite{davies2025hardcore}, stating that the inequality
\begin{equation}\label{eq:occupancy-fraction-lower-bdd-clique}
    \alpha_G(\lambda) \geq \frac{1}{n}\sum_{i=1}^{n} \alpha_{K_{d_i +1}}(\lambda) = \frac{1}{n}\sum_{i=1}^{n}\frac{\lambda}{1+(d_i+1)\lambda}
\end{equation}
holds for all nonnegative $\lambda$ and any simple graph $G$ on $[n]$. It is not hard to see that this conjecture implies the Sah--Sawhney--Stoner--Zhao inequality, i.e.,~\Cref{thm:indep-poly-multiaff-lower-bd} with all $\lambda_v=\lambda$. 
Although the regular case $d_v=d$ was settled by Davies, Jenssen, Perkins, and Roberts~\cite{DJPR_2017} and Davies, Sandhu, and Tan~\cite[Theorem~1.3]{davies2025expectations} proved the irregular case provided $\lambda\leq \frac{3}{(\Delta(G)+1)^2}$, the conjecture remains open. 

Having seen the Davies--Kang conjecture, it seems plausible to look for a strengthening of~\Cref{thm:indep-poly-multiaff-lower-bd} in terms of occupancy fractions.
To elaborate, for \(v\in V(G)\), let \(p_v(\blambda)\defeq \lambda_v \frac{\partial}{\partial{\lambda_v}} \log Z_G(\blambda)\) be the \emph{vertex occupancy} or \emph{vertex marginal} at \(v\). This is the probability that \(v\) is occupied by an independent set randomly sampled according to the multivariate hard-core model.
For \(t\in\mathbb{R}\), let
\[
    \alpha_G(t;\blambda)
    \defeq \frac{t}{\abs{V(G)}} \frac{\partial}{\partial t} \log Z_G(t\blambda)
    = \frac{t}{\abs{V(G)}} \sum_{v\in V(G)} \lambda_v \frac{\partial}{\partial\lambda_v} \log Z_G(t\blambda)
    = \frac{1}{\abs{V(G)}} \sum_{v\in V(G)} p_v(t\blambda).
\]
If \(t=\lambda\) and \(\blambda=(1,\dots,1)\), this is exactly \(\alpha_G(\lambda)\). Then \Cref{thm:indep-poly-multiaff-lower-bd} follows by showing
\begin{equation}\label{eq:avg-occupancy-fraction-lower-bdd-clique}
    \alpha_G(t;\blambda)
    \ge \frac{1}{\abs{V(G)}} \sum_{v\in V(G)} \alpha_{K_{d_v+1}}(t\lambda_v),
    \qquad t\in \RR_{\ge0},\, \blambda\in(\RR_{\ge0})^{V(G)},
\end{equation}
multiplying both sides by \(\abs{V(G)}/t\), and integrating with respect to \(t\) over \((0,1)\). Putting \(t=\lambda\) and \(\blambda=(1,\dots,1)\), we see that \eqref{eq:occupancy-fraction-lower-bdd-clique} is a special case of \eqref{eq:avg-occupancy-fraction-lower-bdd-clique}. Moreover, \eqref{eq:avg-occupancy-fraction-lower-bdd-clique} is equivalent to showing that for all \(\blambda\in(\RR_{\ge0})^{V(G)}\),
\[
    \sum_{v\in V(G)} p_v(\blambda) \ge \sum_{v\in V(G)} \alpha_{K_{d_v+1}}(\lambda_v).
\]

\paragraph{Acknowledgements} We are grateful to Jaeseong Oh for stimulating discussions about~\Cref{conj:main} and also to Matthew Jenssen for helpful comments on the results. We thank Google DeepMind and Tony Feng for running multiple models of Gemini to obtain the proof of~\Cref{lem:key-tech-ineq} as well as outlines to prove~\Cref{thm:semiprop-multiaff-lower-bd}.

\printbibliography

\end{document}